\documentclass[12pt,oneside]{article}
\usepackage{latexsym}
\usepackage{amscd}
\usepackage{amssymb}
\usepackage{amstext}
\usepackage{amsmath}
\usepackage{amsxtra}
\usepackage{latexsym, amsmath, amsxtra, amssymb, amscd, amsthm}	
\usepackage{graphicx}
\usepackage[hang,centerlast,small]{subfigure}
\usepackage[a4paper]{geometry}
\usepackage{hyperref}

\newtheorem{Prop}{Proposition}[section]
\newtheorem{Theo}[Prop]{Theorem}

\newtheorem{Lemm}[Prop]{Lemma}

\newtheorem{Rema}{Remark}

\begin{document}
	\fontsize{12pt}{18pt}\selectfont
	\title{Regularization of Inverse Problems by Filtered Diagonal Frame Decomposition under general source}
	\author{Dang Duc Trong$^{1,2}$, Nguyen Dang Minh$^{3}$, Luu Xuan Thang$^4$, Luu Dang Khoa$^{1,2}$}

	\date{\today}
	\maketitle
		$^{\text{1}}$ Faculty of Maths and Computer Science, University of Science, Ho Chi Minh City, Viet Nam.
		
		 $^{\text{2}}$ Vietnam National University, Ho Chi Minh City, Viet Nam.
		 
		$^{\text{3}}$ Faculty of Fundamental Science, Ho Chi Minh City Open University,  Viet Nam.

		$^{\text{4}}$ Faculty of Natural Science, University of Khanh Hoa, Vietnam. 
	
	\begin{abstract}
Let $X$ and $Y$ be Hilbert spaces, and $\mathbf{K}: \text{dom} \mathbf{K} \subset X \to Y$ a bounded linear operator. This paper addresses the inverse problem $\mathbf{K}x = y$, where exact data $y$ is replaced by noisy data $y^\delta$ satisfying $\|y^\delta - y\|_Y \leq \delta$. Due to the ill-posedness of such problems, we employ regularization methods to stabilize solutions. While singular value decomposition (SVD) provides a classical approach, its computation can be costly and impractical for certain operators. We explore alternatives via Diagonal Frame Decomposition (DFD), generalizing SVD-based techniques, and introduce a regularized solution $x^\delta_\alpha = \sum_{\lambda \in \Lambda} \kappa_\lambda g_\alpha(\kappa_\lambda^2) \langle y^\delta, v_\lambda \rangle \overline{u}_\lambda$. Convergence rates and optimality are analyzed under a generalized source condition $\mathbf{M}_{\varphi, E} = \{ x \in \text{dom} \mathbf{K} : \sum_{\lambda \in \Lambda} [\varphi(\kappa_\lambda^2)]^{-1} |\langle x, u_\lambda \rangle|^2 \leq E^2 \}$. Key questions include constructing DFD systems, relating DFD and SVD singular values, and extending source conditions. We present theoretical results, including modulus of continuity bounds and convergence rates for a priori and a posteriori parameter choices, with applications to polynomial and exponentially ill-posed problems.
	\end{abstract}
	
	
	\section{Introduction}

Let $X$ and $Y$ be Hilbert spaces, and let $\mathbf{K}: \text{dom} \mathbf{K} \subset X \to Y$ be a bounded linear operator. In this paper, we seek a solution $x \in X$ to the inverse problem defined by the operator equation
\begin{equation}
	\label{eq-operator-org}
	\mathbf{K}x = y.
\end{equation}
As is customary, we assume that the exact data $y$ is unavailable, and instead, we are given noisy data $y^\delta$ with a known noise level $\delta$. Specifically, the noise satisfies
\begin{equation}
	\label{ieq-noise-org}
	\| y^\delta - y \|_Y \leq \delta.
\end{equation}
Due to the inherent instability of inverse problems, even a small perturbation in the data can lead to significant errors in the solution, rendering the numerical computation of solutions to the inverse problem challenging. To address this issue, we employ a regularization method for the system defined by equations \eqref{eq-operator-org} and \eqref{ieq-noise-org}. One such regularization approach, based on filtering techniques, has been thoroughly developed in \cite{Kirch} and \cite{Engl}. In this context, $\mathbf{K}$ is assumed to be a compact operator possessing a singular system $(\sigma_k, u_k, v_k)$. Consequently, $\mathbf{K}$ admits a singular value decomposition (SVD) of the form
\begin{equation}
	\label{eq-SVD}
	\mathbf{K} x = \sum_{k=1}^\infty \sigma_k \langle x, u_k \rangle v_k, \nonumber
\end{equation}
where $\sigma_k$ denotes the singular values, and $u_k$ and $v_k$ are singular functions satisfying
\begin{equation}
	\label{eq-SVD2}
	\mathbf{K} u_k = \sigma_k v_k \quad \text{and} \quad \mathbf{K}^* v_k = \sigma_k u_k.
\end{equation}
It is well known that the SVD is a fundamental tool for solving inverse problems. The minimum-norm least-squares solution $x^\ddagger$ to equation \eqref{eq-operator-org} is then given by the Picard formula
\[
x^\ddagger := \mathbf{K}^\ddagger y := \sum_{k=1}^\infty \frac{\langle y, v_k \rangle}{\sigma_k} u_k,
\]
provided the Picard condition holds:
\[
\sum_{k=1}^\infty \frac{|\langle y, v_k \rangle|^2}{\sigma_k^2} < \infty.
\]
When the exact data $y$ is replaced by the noisy data $y^\delta$ with a given noise level $\delta$, the approximate solution takes the form
\begin{equation}
	x^\delta_\alpha := \mathbf{R}_\alpha y^\delta := \sum_{k=1}^\infty \sigma_k g_\alpha (\sigma_k^2) \langle y^\delta, v_k \rangle u_k,
	\label{SVD-regularization}
\end{equation}
where $\alpha > 0$, $g_\alpha$ is a function satisfying $g_\alpha(\lambda) \to 1/\lambda$ as $\alpha \to 0$, and $\mathbf{R}_\alpha$ represents a regularization operator for equation \eqref{eq-operator-org}. Furthermore, the convergence rate and optimality of this regularization are analyzed under the classical source condition $x^\ddagger = \varphi(\mathbf{K}^*\mathbf{K})z$ for some $z \in X$ (see, e.g., \cite{Tautenhahn}).

However, computing the SVD of an operator is often nontrivial and, in certain cases, computationally expensive, as noted by Ebner, Goppel, and Donoho in \cite{Ebner}, \cite{Goppel}, and \cite{Donoho}, respectively. Additionally, SVD-based regularization may not be well-suited for a variety of problems, as highlighted by Donoho in \cite{Donoho}. Thus, developing more efficient computational methods becomes essential. One approach is to identify a system that partially satisfies the SVD conditions in \eqref{eq-SVD2}. A notable development in this direction retains the second condition, i.e., $\mathbf{K}^* v_\lambda = \kappa_\lambda u_\lambda$, where $\lambda$ belongs to a countable index set $\Lambda$. This concept underpins the Wavelet-Vaguelette Decomposition (WVD) in \cite{Donoho}, and more generally, the Diagonal Frame Decomposition (DFD) in \cite{Ebner}, as well as the Translation-Invariant DFD (TI-DFD) in \cite{Goppel}. By leveraging frame theory, such generalizations enable the construction of an expansion analogous to the Picard formula, suggesting that regularization methods tailored to this framework hold significant potential.

Indeed, in \cite{Ebner}, the authors reformulate foundational concepts for DFD-based regularization filtering, akin to the SVD filtering presented in \cite{Engl}. To estimate regularization errors, they adapt the source condition, assuming the solution belongs to a DFD-type source set with a polynomial form:
\begin{equation}
	\label{eq-source}
	\mathbf{M}_{p, E} := \left\{ x \in \text{dom} \mathbf{K} : \sum_{\lambda \in \Lambda} \kappa_\lambda^{-4\nu} |\langle x, u_\lambda \rangle|^2 \leq E^2 \right\}, \nonumber
\end{equation}
where $\nu, E > 0$, and $(u_\lambda, v_\lambda, \kappa_\lambda)$ constitutes a DFD of the operator $\mathbf{K}$. A similarly modified source condition appears in \cite{Hubmer}, where the authors explore the optimality of a posteriori regularization methods. The framework developed in \cite{Ebner, Hubmer} opens up numerous application possibilities. As this theory is still emerging, several natural questions arise:

(i) How can DFD systems be constructed for specific problems?

(ii) What is the relationship between the DFD singular values $\kappa_\lambda$ and the SVD singular values $\sigma_k$? How do these quasi-singular values influence the regularization of ill-posed problems?

(iii) Can the polynomial DFD source condition be generalized to other forms, such as logarithmic source conditions?

(iv) How does the DFD source condition relate to the classical source condition?

(v) Do a priori and a posteriori regularization methods achieve optimality?

Question (i) is particularly compelling and has been extensively explored in the field of tomography (see \cite{Ebner, Hubmer}). This problem exhibits polynomial ill-posedness, where WVD systems prove effective. In Section 4 of this paper, we examine the backward fractional problem, considering two scenarios: polynomial ill-posedness and exponential ill-posedness. In the latter case, the WVD system appears inadequate, prompting us to propose a specialized DFD system.

The investigation of Question (ii) remains in its early stages. In \cite{Ebner, Hubmer}, it is limited to assessing the ill-posedness of the problem $\mathbf{K}x = y$. Our paper advances this inquiry by exploring the ``sparseness'' or ``thickness'' of the DFD singular values $\kappa_\lambda$ through the set
\begin{equation}
	D_{\lambda,\beta} = [\delta_\lambda^*, \beta^{-1} \delta_\lambda^*], \quad \text{where} \quad \delta_\lambda^* = |\mathbf{v}|_{\inf}^{-1} E \sqrt{\kappa_\lambda^2 \varphi(\kappa_\lambda^2)},
	\label{set-D}
\end{equation}
with $\varphi$ being an index function (detailed in the subsequent paragraph). This set facilitates assertions regarding the sequential or uniform optimality of regularization methods.

Inspired by Question (iii), we extend the results of \cite{Ebner, Hubmer} to a DFD source set defined by a general source function, rather than a polynomial one. Specifically, for a positive constant $E$,
\begin{equation}
	\label{eq-source-1}
	\mathbf{M}_{\varphi, E} := \left\{ x \in \text{dom} \mathbf{K} : \sum_{\lambda \in \Lambda} [\varphi(\kappa_\lambda^2)]^{-1} |\langle x, u_\lambda \rangle|^2 \leq E \right\},
\end{equation}
where the index function $\varphi$ satisfies conditions outlined in later assumptions. Such conditions naturally arise in ill-posed problems, such as tomography with $\varphi(\mu) = \mu^{2\nu}$ or the backward problem with $\varphi(\mu) = (-\ln \mu)^{-p}$ (see subsequent sections). This topic merits further attention, and in Subsection 4.4 of our paper, we study the latter index function $\varphi$.

To address Question (iv), we present two examples demonstrating that the classical source condition can suffice to derive the DFD source condition. These examples illustrate the connection between classical and DFD source conditions, though these findings are preliminary and warrant deeper investigation in future work.

Question (v) is thoroughly explored in this paper. Building on \cite{Ebner, Hubmer} and the framework of \eqref{SVD-regularization}, we construct a filtered regularization of the form
\begin{equation}
	\label{eq-solution-regular-noise-dfd}
	x^\delta_\alpha := R_\alpha y^\delta = \sum_{\lambda \in \Lambda} \kappa_\lambda g_\alpha (\kappa_\lambda^2) \langle y^\delta, v_\lambda \rangle_Y \overline{u}_\lambda.
\end{equation}
Additionally, we enhance the analysis by addressing a posteriori strategies, filling a gap left by \cite{Ebner} and \cite{Goppel}. We further refine the classification of optimality properties, distinguishing between sequential order optimality (as noted in \cite{Ebner, Hubmer}) and global order optimality, the latter of which has not been previously addressed.

In terms of structure, Section 2 reviews foundational results on frames and defines optimal regularization. Section 3 presents the main results of this paper, including lower bounds for the modulus of continuity of $\mathbf{K}^{-1}$ on the set $\mathbf{M}_{\varphi,E}$ and convergence rates for a priori and a posteriori parameter choices. Section 4 applies these theoretical findings to specific examples, while Section 5 provides the proofs of the main results.
	
	\section{Some basic notions and notations}
	\subsection{Notions of frames}
	Letting $\Lambda$ be an at most countable set of indices, we denote
	$$ l^2(\Lambda)=\left\{\mathbf{a}=(a_\lambda)_{\lambda\in\Lambda}: a_\lambda\in \mathbb{R},
	\sum_{\lambda\in\Lambda}|a_\lambda|^2<\infty\right\} $$
	with the norm $\|\mathbf{a}\|_{2}=\left(\sum_{\lambda\in\Lambda}|a_\lambda|^2\right)^{1/2}$.
	
	Before delving into the specific content of the article, we would like to recall some results about frames in a Hilbert space $\mathbb{K}$. These results can be found in \cite{Christ}, \cite{Ebner}, and \cite{Hubmer}. For convenience, let us introduce the definition of a frame.

	\noindent{\bf Definition.} {\it A sequences $\mathbf{w}=\{w_\lambda\}_{\lambda \in \Lambda} $ in a Hilbert subspace $\mathbb{H}\subset \mathbb{K}$ is called a frame over $\mathbb{H}$, if and only if there exists frame bounds $0< A_w, B_w \in \mathbb{R}$ such that for all $x \in \mathbb{H}$ there holds
		\begin{equation}
			A_w \left\|x \right\|^2_{\mathbb{K}} \leq \sum_{\lambda\in\Lambda} \left|\langle x, w_\lambda \rangle_{\mathbb{K}} \right|^2 \leq B_w \left\|x \right\|^2_{\mathbb{K}}.
			\label{AwBw} 	
		\end{equation}
		If 
		$w_{\lambda_0}\not\in\overline{{\rm span}\{w_\lambda\}}_{\lambda\not=\lambda_0}$  for every $\lambda_0\in \Lambda$ then we say that the frame is minimal.
		
		For convenience, we denote 
		$\|x\|_w:=\sqrt{\sum_{\lambda\in\Lambda} \left|\langle x, w_\lambda \rangle_{\mathbb{K} }\right|^2} $ for every $x\in \mathbb{K}$ and 
		\begin{align*}
			|\mathbf{w}|_{\inf}&=\inf\{\|x\|_w: x\in \mathbb{H} ~\text{and}~\|x\|_{\mathbb{K}}=1\},\\
			|\mathbf{w}|_{\sup}&=\sup\{\|x\|_w: x\in \mathbb{H} ~\text{and}~\|x\|_{\mathbb{K}}=1\}.  
		\end{align*}
		If $|\mathbf{w}|_{\inf}=|\mathbf{w}|_{\sup} $, we say that the frame is tight and denote $|\mathbf{w}|_{\rm fr}:=|\mathbf{w}|_{\sup}=|\mathbf{w}|_{\inf}$ .
		
	}
	
	From the definition, we have $0<\sqrt{A_w}\leq |\mathbf{w}|_{\inf}\leq |\mathbf{w}|_{\sup}\leq \sqrt{B_w} $ and
	\begin{equation}
		|\mathbf{w}|_{\inf} \|x\|_{\mathbb{K}}\leq \|x\|_w\leq |\mathbf{w}|_{\sup}\|x\|_{\mathbb{K}}
		\label{AwBw-supinf}    
	\end{equation} 
	for every $x\in\mathbb{H}$.
	For $x'\in \mathbb{K}$, we have
	\begin{equation}
		\|x'\|_w^2=\sum_{\lambda\in\Lambda} \left|\langle P_{\mathbb{H}}x', w_\lambda \rangle_{\mathbb{K}} \right|^2 \leq |\mathbf{w}|_{\sup} \left\|P_{\mathbb{H}} x' \right\|^2_{\mathbb{K}}
		\leq |\mathbf{w}|_{\sup} \left\|x' \right\|^2_{\mathbb{K}}.
		\label{Bw-general}
	\end{equation}
	Here $P_{\mathbb{H}}$ is the orthogonal projection on $\mathbb{H}$.
	For a given frame $\{w_\lambda\}_{\lambda\in\Lambda} $ , one can define the frame analysis operator $F$ as below 
	\begin{equation}
		\label{eq-analysis-operator}
		F: \mathbb{H} \to l_2 \left( \Lambda\right), \quad  x \mapsto \left\{\langle x, w_\lambda \rangle_{\mathbb{K}} \right\}_{\lambda\in\lambda}.
		\nonumber
	\end{equation}
	And, the synthesis operator $F^*$, which is given by 
	\begin{equation}
		\label{eq-synthesis-operator}
		F^*: l_2\left( \Lambda\right) \to {\mathbb{H}} , \quad (a_\lambda)_{\lambda\in \Lambda} \mapsto \sum_{\lambda\in\Lambda} a_\lambda w_\lambda.
		\nonumber
	\end{equation}
	From the inequality \eqref{AwBw}, there holds
	\begin{equation}
		\label{ieq-bounded-operator-norm}
		\sqrt{A_w} \leq \Vert F \Vert = \Vert F^* \Vert \leq \sqrt{B_w}. \nonumber
	\end{equation}
	We can define the operator $S: = F^*F$, that is, 
	$$Sx: = \sum_{\lambda\in\Lambda} \langle x, w_\lambda \rangle_{\mathbb{K}} w_\lambda.$$
	It is worth noting that, in this case, the operator $S$ is a bounded, linear, and invertible operator. Specifically, $A_w I \leq S \leq B_w I$ và $B_w^{-1}I \leq S^{-1} \leq A_w^{-1} I$. Therefore, if we set $\overline{w}_k: = S^{-1} w_k$ then one holds
	\begin{equation}
		B_w^{-1} \left\|x \right\|^2_{\mathbb{K}} \leq \|x\|^2_{\overline{w}} \leq A_w^{-1} \left\|x \right\|^2_{\mathbb{K}} \nonumber
	\end{equation}
	for every \(x \in \mathbb{H}\). Consequently, the set \(\{\overline{w}_\lambda\}_{\lambda\in \Lambda}\) is also a frame over \(\mathbb{H}\). 
	As we know, it is referred to as the dual frame of \(\{w_k\}_{k\in \mathbb{N}}\). In that case, the analysis and synthesis operators of this frame are as follows.
	The analysis operator $\overline{F}$ is defined as below 
	\begin{equation}
		\label{eq-analysis-operator}
		\overline{F}: \mathbb{H} \to l_2 \left( \Lambda\right), \quad  x \mapsto \left\{\langle x, \overline{w}_\lambda \rangle_\mathbb{K} \right\}_{\lambda \in \Lambda}
		\nonumber
	\end{equation}
	and, the synthesis operator $\overline{F}^*$, which is given by 
	\begin{equation}
		\label{eq-synthesis-operator}
		\overline{F}^*: l_2\left( \Lambda\right) \to \mathbb{H} , \quad \{a_\lambda\}_{\lambda
			\in \Lambda} \mapsto \sum_{\lambda\in\Lambda} a_\lambda \overline{w}_\lambda. \nonumber
	\end{equation}
	From the inequalities \eqref{AwBw}, \eqref{AwBw-supinf} there also holds
	\begin{equation}
		\label{ieq-bounded-operator-norm}
		\sqrt{B_w^{-1}}\leq |\mathbf{w}|_{\sup}^{-1} \leq \Vert \overline{F} \Vert = \Vert \overline{F}^* \Vert \leq |\mathbf{w}|_{\inf}^{-1}\leq \sqrt{A_w^{-1}}. \nonumber
	\end{equation}
	It follows that
	\begin{equation}
		\left\| \sum_{\lambda\in \Lambda} a_\lambda \overline{w}_\lambda\right\|_{\mathbb{K}}
		=\|\overline{F}^*(\{a_\lambda\})\|_\mathbb{K}\leq  |\mathbf{w}|_{\inf}^{-1}\left(\sum_{\lambda\in\Lambda}|a_\lambda|^2\right)^{1/2}.
		\label{w-dual-upper-inequality}
	\end{equation}
	Moreover, it can also be proved that
	$ \overline{F}^* F = F^*\overline{F} = I, $
	and thus, for any \(x \in \mathbb{H}\), it can always be expressed as $x = \sum_{\lambda\in\Lambda} x_\lambda  \overline{w}_\lambda$ where $x_\lambda=\langle x, w_\lambda \rangle_\mathbb{H}+a_\lambda$ with $\mathbf{a}=(a_\lambda)\in N(\overline{F}^*)$. Especially, we have 
	\begin{equation}
		\label{eq-present-x}
		x = \sum_{\lambda\in\Lambda} \langle x, w_\lambda \rangle_\mathbb{H}  \overline{w}_\lambda. 
	\end{equation}
	Generally, the calculation of $\overline{w}_\lambda$ is only easy in some special case. In fact, if $\{w_\lambda\}$ is tight  then $\overline{w}_\lambda=\frac{1}{|\mathbf{w}|_{\rm fr}}w_\lambda$ (see, e.g., \cite{Christ}, chap. 5).
	In general, we always have $\{0\} \subset N\left(F^*\right) = N\left(\overline{F}^*\right)$, and therefore, the representation of $x$ in (\ref{eq-present-x}) is not unique. However, this representation is considered the most economical according to \cite{Hubmer}. From \cite{Ebner}, we have known that the frame $\{\overline{w}_\lambda\}$ is the biorthonormal sequence of $\{w_\lambda\}$, i.e. $\langle w_\lambda,\overline{w}_\nu\rangle_{\mathbb{K}}=\delta_{\lambda\nu}$ for $\lambda,\nu\in\Lambda$, is equivalent to $\{w_\lambda\}$ being minimal. 
	In this case, we have $x_\lambda=\langle x, w_\lambda \rangle_\mathbb{H}$ and the expansion \eqref{eq-present-x} is unique.    
	Next, we recall the definition of diagonal frame decomposition (see, e.g., \cite{Ebner}). 
	
	\noindent{\bf Definition.} {\it
		Let $\mathbf{K}: X \to Y$ be bounded linear operator, and $ \Lambda$ is an at most countable index set. We set $\left(\mathbf{u}, \mathbf{v}, \boldsymbol{\kappa} \right)= \left(u_\lambda, v_\lambda, \kappa_\lambda \right)_{\lambda \in \Lambda}$ a diagonal frame decomposition (DFD) for the operator $\mathbf{K}$ if the following holds
		\begin{description}
			\item[(D1)] $\{u_\lambda\}_{\lambda \in \Lambda}$ is a frame over $\left( ker \mathbf{K}\right)^\bot \subset X$.
			\item[(D2)]  $\{v_\lambda\}_{\lambda \in \Lambda}$ is a frame over $ \overline{ran \mathbf{K}} \subset Y$.
			\item[(D3)]  $\left(\kappa_\lambda \right)_{\lambda \in \Lambda} \in \left( 0; \infty \right)^\Lambda$ satisfies the quasi-singular relations
			\begin{equation}
				\label{eq-dfd-condition}
				\mathbf{K}^* v_\lambda = \kappa_\lambda u_\lambda, \text{ for all } \lambda \in \Lambda. \nonumber
			\end{equation}
		\end{description}
		The $\kappa_\lambda$ are called the DFD singular value.
	}
	
	For $h:\mathbb{R}\to [0,\infty)$, $x\in X$,  we define 
	$$ \langle h(\mathbf{K}^*\mathbf{K})x,x \rangle_u=\sum_{\lambda\in\Lambda}h(\kappa_\lambda^2)
	|\langle x,u_\lambda\rangle_Y|^2. $$

	From ({\bf D1}), ({\bf D2}) we can find numbers $A_u, A_v, B_u, B_v>0$ such that
	\begin{align}
		A_u\|w\|^2_X\leq \|w\|_u^2\leq B_u\|w\|^2_X, & \forall w\in (ker \mathbf{K})^\perp,
		\label{AuBu} \\
		A_v\|z\|^2_Y\leq \|z\|^2_v \leq B_v\|z\|^2_Y, & \forall z\in \overline{rank\mathbf{K}}.
		\label{AvBv}
	\end{align}
	From now on, we always denote by $a^*$ an extended real  number such that $a^*>\sup_\lambda\kappa_\lambda^2$
	if $\sup_\lambda\kappa_\lambda^2<\infty$ and $a^*=\infty$ if $\sup_\lambda\kappa_\lambda^2=\infty$. 
	
	\subsection{Notions of the worst case error and optimality}
	Consider the problem \eqref{eq-operator-org} and denote the Moore-Penrose operator 
	$$ \mathbf{K}^\ddagger (z)= {\rm argmin} \{\|h\|_{X}: h\in~
	{\rm dom}~\mathbf{K}, z\in {\rm rank} \mathbf{K}, \mathbf{K}(h)=z\}. $$
	We denote
	the Moore-Penrose solution 
	of \eqref{eq-operator-org} by $x^\ddagger=\mathbf{K}^\ddagger y$.
	Let
	an operator $\mathbf{R}:Y\to X$ satisfy $\mathbf{R}y\approx x$. We say that $\mathbf{R}$ is an approximation method of the problem 
	\eqref{eq-operator-org}. Assume that the solution $x^\ddagger$ of \eqref{eq-operator-org} belongs to a subset $\mathbf{M}\subset X$, we recall the definition of the worst-case error of the method $\mathbf{R}$ on $\mathbf{M}$ as below.
	$$
	\Delta (\mathbf{M}, \delta, \mathbf{R}):=\sup \left\{\|\mathbf{R} y^\delta-x^\ddagger\|: x^\ddagger \in \mathbf{M} \wedge y^\delta \in Y \wedge \| \mathbf{K} x^\ddagger-y^\delta \| \leq \delta\right\}.
	$$
	We say that the method $\mathbf{R}_{opt}:Y\to X$ is optimal on $\mathbf{M}$ if 
	$\Delta (\mathbf{M}, \delta, \mathbf{R}_{opt})= \inf_{\mathbf{R}}\Delta (\mathbf{M}, \delta, \mathbf{R}) $ and $\mathbf{R}_{opt}:Y\to X$ is order optimal on $\mathbf{M}$ if there is a $c>0$ independent of $\delta$ such that
	$\Delta (\mathbf{M}, \delta, \mathbf{R}_{opt})\leq c\inf_{\mathbf{R}}\Delta (\mathbf{M}, \delta, \mathbf{R}).  $
	In  our paper, we choose $\mathbf{M}=\mathbf{M}_{\varphi,E}$ defined in (\ref{eq-source-1}).
	Let  $\mathbf{R}_\alpha: Y\to X$, $\alpha>0$, be a family of bounded operators and let
	$\alpha^*:(0,\alpha_0)\times Y\to (0,\infty)$. As in \cite{Kirch, Ebner}, we say that
	$(\mathbf{R}_\alpha,\alpha^*)$  is a regularization method if 
	\begin{align*}
		&\limsup_{\delta\to 0^+}\{\alpha^*(\delta,y^\delta):~ y^\delta\in Y\wedge \|y^\delta-y\|_Y\leq\delta\}=0,\\
		&\limsup_{\delta\to 0^+}\{\|\mathbf{K}^\ddagger y-\mathbf{R}_{\alpha^*(\delta,y^\delta)}\|:~
		y^\delta\in Y\wedge \|y^\delta-y\|_Y\leq \delta\}=0.
	\end{align*}
	The quantities $\alpha$ and $\alpha^*$ are called the regularization parameter  and the admissible parameter choice respectively.

	Inspired from the classical optimal regularization theory
	(\cite{Tautenhahn, Vainik}), we can
	classify the order optimality for our problem.
	
	\noindent{\bf Definition.} {\it We say that the regularization method $(\mathbf{R}_\alpha,\alpha^*)$ is 
		
		(a) sequential order optimal if there a sequence $\delta_n\to 0^+$ such that there exists a constant $c>0$ independent of $n$ such that
		$\Delta (\mathbf{M}, \delta_n, \mathbf{R}_{\alpha^*(\delta_n,y^{\delta_n})})\leq c\inf_{\mathbf{R}}\Delta (\mathbf{M}, \delta, \mathbf{R})$,
		
		(b) uniform order optimal if there is a $\delta_0$ and a constant $c$ in dependent of $\delta$ such that  $\Delta (\mathbf{M}, \delta, \mathbf{R}_{\alpha^*(\delta,y^{\delta})})\leq c\inf_{\mathbf{R}}\Delta (\mathbf{M}, \delta, \mathbf{R})$ for every $\delta\in(0,\delta_0)$. 
	} 
	
	Sequential optimal regularization are studied in 
	the recent papers \cite{Ebner, Goppel, Hubmer}, but uniformly optimal regularizations are not yet discussed.  
	
	
	\section{Main results}
	\subsection{Pointwise convergence}
	
	Let $\left(\textbf{u}, \textbf{v}, \boldsymbol{\kappa} \right)  $ be a DFD for $\mathbf{K}$ and 
	$y$ be as in \eqref{eq-operator-org}, we have
	$$ \langle y,v_\lambda\rangle_Y=\langle \mathbf{K}x^\ddagger,v_\lambda\rangle_Y=\langle x^\ddagger,\mathbf{K}^*v_\lambda\rangle_X =\langle x^\ddagger,\kappa_\lambda u_\lambda\rangle_X~~~{\rm for}~\lambda\in\Lambda.$$ 
	Hence, from the expansion \eqref{eq-present-x},
	the Moore-Penrose solution of (\ref{eq-operator-org}) has the expansion 
	\begin{equation}
		\label{eq-solution-dfd}
		x^\ddagger: = \mathbf{K}^\ddagger y = \sum_{\lambda \in \Lambda} \dfrac{1}{\kappa_\lambda} \langle y, v_\lambda \rangle_Y \overline{u}_\lambda.
	\end{equation}
	The expansion implies $y\in {\rm dom} \mathbf{K}^\ddagger$ if and only if $\sum_\lambda \left|\frac{\langle y,v_\lambda\rangle_Y}{\kappa_\lambda}\right|^2<\infty$.
	The stability of solution \eqref{eq-solution-dfd} is depended on the infimum of $\{\kappa_\lambda\}$.
	In fact we have
	\begin{Theo}
		\label{well-posed}
		Assume that $\inf_{\lambda\in\Lambda}\kappa_\lambda\geq\kappa_0>0$ and \eqref{ieq-noise-org} holds. Then  the operator
		$\mathbf{K}^\ddagger: {\rm rank} \mathbf{K}\to X$
		can be extended to the linear bounded operator $\overline{\mathbf{K}}^\ddagger:Y\to X$  which has the formula
		$$\overline{\mathbf{K}}^\ddagger h = \sum_{\lambda \in \Lambda} \dfrac{1}{\kappa_\lambda} \langle h, v_\lambda \rangle_Y \overline{u}_\lambda, ~~~ \text{for}~h\in Y$$
		and
		$$ \|\overline{\mathbf{K}}^\ddagger g^\delta-x^\ddagger\|_X\leq \frac{|\mathbf{v}|_{\sup}}{\kappa_0|\mathbf{u}|_{\inf}}\delta. $$
	\end{Theo}
	\begin{Rema}
		If $\{u_\lambda\}$ is tight then we obtain 
		\begin{equation}
			x^\ddagger: = \mathbf{K}^\ddagger y = \sum_{\lambda \in \Lambda} \dfrac{1}{\kappa_\lambda|\mathbf{u}|_{\rm fr}} \langle y, v_\lambda \rangle_Y {u}_\lambda.
			\nonumber
		\end{equation}
	\end{Rema}
	In the case $\inf_{\lambda\in\Lambda}\kappa_\lambda=0 $, the equality \eqref{eq-solution-dfd} could be instability with respect to $g$. Suggested by the classical regularization \eqref{SVD-regularization}, we can construct the regularization solution for the equation (\ref{eq-operator-org}) - (\ref{ieq-noise-org}) in the form \eqref{eq-solution-regular-noise-dfd}
	where $g_\alpha:[0,a^*)\to\mathbb{R}$ is filter functions that satisfy 
	
	{\bf Assumption C}
	\begin{description}
		\item[C1] For all $\alpha>0$, $\mu\in [0,a^*)$: $\sqrt{\mu} g_\alpha \left(\mu \right) < \infty$. 
		\item[C2] There exists a constant $C_g>0$ such that $\sup\{\left|\mu g_\alpha \left(\mu \right) \right|: \alpha>0, 0\leq  \mu
		< a^* \} \leq C_g$.
		\item[C3] For all $\mu \in \left(0, a^* \right)$ there holds $\lim_{\alpha \to 0} \mu g_\alpha \left(\mu \right) = 1$.
	\end{description}
	As shown in \cite{Ebner}, we have
	\begin{Theo}
		\label{convergence}
		Let Assumption C and \eqref{ieq-noise-org} hold.
		Then for every $\alpha(\delta)\to 0$ as $\alpha\to 0^+$, we have
		$$  \lim_{\delta\to 0^+}\|x^\delta_{\alpha(\delta)}-x^\ddagger\|_X=0. $$
	\end{Theo}

	\subsection{Lower bound of worst case error}

	To evaluate the optimality of the proposed regularization $R_\alpha $  we first find a lower bound for the worst-case error of the regularization algorithms.
	This will be useful for proving the optimality of the DFD-based regularization method  over the source set $ \mathbf{M}_{\varphi, E}$ in later theorems. It also serves as a basis for choosing appropriate regularization parameters.
	
	Similar to \cite{Engl}, we shall consider the computation of the worst-case error of the regularization operator $\mathbf{R}:Y\to X$ in the source set $\mathbf{M}_{\varphi,E}$ with the function $\varphi$ satisfying  
	\textbf{Assumption A1.} Function $\varphi:(0, a^*) \rightarrow(0, \infty)$  is continuous and satisfies the following conditions
	\begin{description}
		\item[(i)] $\lim _{\mu \rightarrow 0} \varphi \left(\mu \right)=0$,
		\item[(ii)] Function $\varphi$ that is strictly monotone increasing in $\left(0, a^*\right)$,
		\item[(iii)] Function $\Theta \left(0, \varphi\left(a^*\right)\right] \to \left(0, a^*\varphi\left(a^*\right)\right]$  is defined by  $\Theta\left(\mu\right)=\mu \varphi^{-1}\left(\mu\right)$ and that is a convex.
	\end{description}
	Here we denote $\varphi(a^*)=\lim_{\mu\to a^{*-}}\varphi(\mu)$.
	We consider  the scenario where the frame $\mathbf{u}$ admits a biorthogonal sequence $\overline{\mathbf{u}}=\left(\overline{u}_\lambda\right)_{\lambda \in \Lambda}$ with $\forall \lambda, \nu \in \Lambda:\left\langle u_\lambda,\overline{u}_\nu\right\rangle=\delta_{\lambda \nu}$. Similar to the result about the lower bound in classical theory (see \cite{Tautenhahn}), we have
	\begin{Theo} Let $\delta_0>0$, $\delta\in(0,\delta_0)$, $\beta\in (0,1)$ and let $(\mathbf{u}, \mathbf{v}, \boldsymbol{\kappa})$ be a DFD of $\mathbf{K}$ such that $\mathbf{u}$ is minimal and $\inf_{\lambda\in\Lambda}\kappa_\lambda=0$. Let the set $D_{\lambda,\beta}$ be as in \eqref{set-D}. 	
		For the source sets $\mathbf{M}_{\varphi, E}$ defined by (\ref{eq-source-1}), if 
		$\delta\in \bigcup_{\lambda\in\Lambda}D_{\lambda,\beta}$
		then
		\begin{equation}
			\inf_{\mathbf{R}}\Delta(\mathbf{M}_{\varphi, E},\delta,\mathbf{R}) \geq \beta |\mathbf{u}|_{\sup}^{-1}
			E \sqrt{\Theta^{-1}\left(|\mathbf{v}|_{\inf}^2\delta^2 / E^2  \right)}.
			\label{order-lower-bound}    
		\end{equation}
		Especially, if  $(0,\delta_0]\subset \cup_{\lambda\in\Lambda}D_{\lambda,\beta}$ then \eqref{order-lower-bound} holds for every $0<\delta\leq\delta_0$.  
		\label{Theo:lower-bound}
	\end{Theo}
	This theorem serves as a criterion for determining the order optimality of regularizations, so we will have a bit more commentary on it.
	\begin{Rema}
		(i) The condition that the system $\{u_\lambda\}$ is minimal is essential in the proof of the theorem. The investigation of the lower bound when $\{u_\lambda\}$ is not minimal is a worthy topic of study.
		
		(ii) To show that a regularization $R:Y\to X$ is order-optimal, we only need to verify that
		\begin{equation}
			\Delta(\mathbf{M}_{\varphi, E},\delta,\mathbf{R}) \leq C
			E \sqrt{\Theta^{-1}\left(|\mathbf{v}|_{\inf}^2\delta^2 / E^2  \right)}.
			\nonumber    
		\end{equation}
		(iii) In fact, we can prove that 
		$\inf_{\mathbf{R}}\Delta(\mathbf{M}_{\varphi, E},\delta,\mathbf{R}) \geq \beta \sqrt{B_u^{-1}}
		E \sqrt{\Theta^{-1}\left(A_v\delta^2 / E^2  \right)} $
		for all $B_u,A_v$ satisfy \eqref{AuBu}, \eqref{AvBv}. Since $|\mathbf{u}|_{\sup}^{-1}
		\sqrt{\Theta^{-1}\left(|\mathbf{v}|_{\inf}^2\delta^2 / E^2  \right)}\geq \sqrt{B_u^{-1}}E \sqrt{\Theta^{-1}\left(A_v\delta^2 / E^2  \right)} $,
		our lower bound is better.

		(iv) In \cite{Ebner}, to obtain the lower bound of the worst case error, the authors choose $\delta=\delta_\lambda=   
		\sqrt{A_v^{-1}}E\kappa_\lambda^{2\nu+1}$. The case that the mentioned paper examines corresponds to considering the source function $\varphi(\mu)=\mu^{2\nu}$. In this case, we have $\mu\varphi(\mu)=\mu^{2\nu+1}$ and $\delta_\lambda=\sqrt{A_v^{-1}}E\sqrt{\kappa_\lambda^2\varphi(\kappa_\lambda^2)} $. For $\beta=\sqrt{A_v}/|\mathbf{v}|_{\inf}$, since 
		$\sqrt{A_v}\leq |\mathbf{v}|_{\inf}$, 
		we have $0<\beta\leq 1$ and $ \delta_\lambda^*\leq \delta_\lambda\leq \beta^{-1}\delta_\lambda^*$ which gives  
		$\delta_\lambda\in \bigcup_{\lambda\in\Lambda}D_{\lambda,\beta}$. Hence, the inequality
		\eqref{order-lower-bound} hold for the chosen sequence $(\delta_\lambda)_{\lambda\in\Lambda}$. 
		
		(v) As shown in classical optimal regularization theory
		(\cite{Tautenhahn, Vainik}), the optimal property is not true if the singular values of the operator $\mathbf{K}$ are too sparse, e.g., $\lim_{n\to\infty}\sigma_{n+1}/\sigma_n=0$. The distribution of the singular values affects the classification of the optimization types. 
		Similarly, the optimal result depends on the distribution of $\delta_\lambda^*$. In fact, we have $(0,\delta_0)\subset \cup_{\beta>0}\cup_{\lambda\in\Lambda}D_{\lambda,\beta}$. If 
		$(0,\delta_0)\not\subset \cup_{\lambda\in\Lambda}D_{\lambda,\beta} $ for every $\beta>0$ then
		the distribution of $\delta_\lambda^*$  is very sparse. In this case, the lower bound may be valid for only some subsequences of $\delta_\lambda$.
		
		(vi) Note that, in the case of H\"older-type source condition, i.e., $\varphi \left( \mu \right) = \mu^{2\nu}$, $\mu,\nu>0$,  then $\Theta^{-1}(\mu)=\mu^{\frac{2\nu}{2\nu+1}}$. This gives $\Theta^{-1}\left(|\mathbf{v}|_{\inf} \delta^2 / E^2  \right) = |\mathbf{v}|_{\inf}^\frac{2\nu}{2\nu+1} \delta^\frac{4\nu}{2\nu+1} E^\frac{-4\nu}{2\nu+1}$. So we get that 
		$$  \inf_{\mathbf{R}}\Delta(\mathbf{M}_{\varphi, E},\delta,\mathbf{R}) \geq \beta\dfrac{|\mathbf{v}|_{\inf}^{\frac{\nu}{2\nu+1}}}{|\mathbf{u}|_{\sup}} \delta^\frac{2\nu}{2\nu+1} E^\frac{1}{2\nu+1}. $$
		A similar lower bound is stated  in  Theorem 3.11 in \cite{Ebner} with $\delta$ is in the sequence $(\delta_\lambda)_{\lambda\in\Lambda}$ as in Remark (ii). 
		
		(vii) In some problems, we have the logarithm source condition $\varphi(\mu)=(-\ln \mu)^{-p}$, $p>0$. In this case
		$\Theta(\mu)=\mu e^{-\mu^{-1/(2p)}}$ and
		$\sqrt{\Theta^{-1}(\mu)}=\varphi(\mu)(1+ o(1))$ (see, e.g., \cite{Hohage}).  So
		$$  \inf_{\mathbf{R}}\Delta(\mathbf{M}_{\varphi, E},\delta,\mathbf{R}) \geq \beta 
		\sqrt{B_u^{-1}}E\left(\ln\left(\frac{E^2}{|\mathbf{v}|_{\inf}^2 \delta^2}\right)\right)^{-p} (1+o(1)). $$    
	\end{Rema}

	\subsection{Convergence rate and a priori parameter choice}
	Returning to the main content of this article, to extend the results of Ebner and colleagues \cite{Ebner} from a polynomial source set to a more general source set, we consider a source function \(\varphi\) in the set $\mathbf{M}_{\varphi,E}$ defined in \eqref{eq-source-1}. Next, we investigate issues such as the lower bound of the worst-case error, convergence rate in both the choice of a priori and a posteriori parameters. In particular, we have the following theorem 
	
	Moreover, the function $\varphi$ and $g_\alpha$ are also fullfilled some conditions. 
	
	\textbf{Assumption A2.} There are  constants   $\gamma_1,\gamma_2>0$ such that
	\begin{description}
		\item[(i)] $\sup _{0\leq \mu <a^*}\left|\sqrt{\mu} g_\alpha(\mu)\right| \leq \frac{\gamma_1}{\sqrt{\alpha}}$,
		\item[(ii)] $\sup _{0\leq \mu <a^*}\left|1-\mu g_a(\mu)\right| \sqrt{\varphi(\mu)} \leq \gamma_2 \sqrt{\varphi(\alpha)}$.
	\end{description}
	
	Next we derive convergence rates which give quantitative estimates on the reconstruction error $\left\|x^{\ddagger}-x_{\alpha}^\delta\right\|_X$. Recall that, in this section, we assume that the source function $\varphi$ and the filter function $g_\alpha$ satisfy Assumption A1 and Assumption A2.
	\begin{Theo}
		\label{Theo:prior}
		Let $A_v\in (0,|\mathbf{v}|_{\inf}^2)$. For $(\mathbf{u}, \mathbf{v}, \boldsymbol{\kappa})$ being a DFD  of $\mathbf{K}$, with $\overline{\mathbf{u}}$ as a dual frame of $\mathbf{u}$ and $x^\ddagger \in \mathbf{M}_{\varphi, E}$. In this case, if we choose the regularization parameter as
		\begin{equation}
			\label{eq-parameter-choose-prior}
			\alpha(\delta)=\alpha^*\left(\delta, y^\delta\right):=\varphi^{-1}\circ \Theta^{-1}\left(A_v\delta^2 / E^2\right),
		\end{equation}
		then
		the following convergence rate result holds:
		\begin{equation}
			\label{ieq-prior}
			\left\|x^\delta_{\alpha(\delta)} -  x^{\ddagger}\right\|_X \leq \sqrt{A_u^{-1}A_v^{-1}} \left(\gamma_1 \sqrt{B_v} +  \gamma_2\sqrt{A_v}  \right)E \sqrt{\Theta^{-1}\left(|\mathbf{v}|_{\inf}^2\delta^2 / E^2\right)},
		\end{equation}
		where $A_u, B_v$ are bounds of $\mathbf{u}$ and $\mathbf{v}$, respectively. And $\gamma_1, \gamma_2$ be constants as in Assumption A2.  From the inequality \eqref{ieq-prior} we obtain
		$$  \Delta(\mathbf{M}_{\varphi,E},\delta,R_{\alpha(\delta)})\leq \sqrt{A_u^{-1}A_v^{-1}} \left(\gamma_1 \sqrt{B_v} +  \gamma_2\sqrt{A_v}  \right)E \sqrt{\Theta^{-1}\left(|\mathbf{v}|_{\inf}^2\delta^2 / E^2\right)} ,$$
		where $R_\alpha$ is defined in \eqref{eq-solution-regular-noise-dfd}.  
		
		(i) If  $\mathbf{u}$ is minimal then $R_{\alpha(\delta)}$ is sequential order optimal. 
		
		(ii) For a fixed $\beta\in (0,1)$, if  $\mathbf{u}$ is minimal and $(0,\delta_0]\subset\bigcup_{\lambda\in\Lambda}
		D_{\lambda,\beta}$, then $R_{\alpha(\delta)}$ is uniform order optimal. Here $D_{\lambda,\beta}$ is defined in 
		\eqref{set-D}.
	\end{Theo}
	\begin{Rema}
		(i) Note that, in the case $\varphi\left(\mu \right) = \mu^{2\nu}, p>0$ then our result will become a similar result to Theorem 3.8 in \cite{Ebner} and Theorem 2.5 in \cite{Hubmer}.
		
		(ii)  Calculating the exact number $|\mathbf{v}|_{\inf}$ is not easy. Therefore, choosing $A_v$ as in the theorem will make the calculation of $\alpha(\delta)$ more feasible. However, if $A_v$ is small, the error will contain $A_v^{-1}$ and so will be large. Therefore, $A_v$ should be chosen such that $\beta|\mathbf{v}|_{\inf}\leq A_v\leq |\mathbf{v}|_{\inf}$.	
	\end{Rema}
	
	\subsection{Posteriori parameter choice}
	In this subsection, we present the results of the discrepancy between the exact solution and the regularized solution by choosing the posterior regularization parameter according to the Morozov principle. Specifically, the chosen parameter for posterior regularization $\alpha$ following the Morozov principle \cite{193}, this principle helps find the parameter $\alpha$ as a solution to the following equation
	\begin{equation}
		\label{eq-choosing-alp-posterior}
		d(\alpha):=\left( \sum_{\lambda\in\Lambda} \left| \langle \mathbf{K}  x_\alpha^\delta-y^\delta, v_\lambda \rangle_Y \right|^2 \right)^\frac{1}{2}=\tau\sqrt{B_v}\delta \quad\text { with } \quad \tau> 1 .
	\end{equation}
	
	Let's additionally assume that the function $g_\alpha$ satisfies the following Assumption B1.
	
	\textbf{Assumption B1}. The function  $g_\alpha:(0, a^*] \rightarrow(0, \infty)$  satisfies 
	\begin{description}
		\item[(i)] $1-\mu g_\alpha(\mu) \rightarrow \rho$ for $\rho>0$, $\alpha \rightarrow \infty$ and for each $\mu \in[0, a^*),$
		\item[(ii)] $g_{\alpha_n}(\mu) \rightarrow g_\alpha(\mu)$ for $\alpha_n \rightarrow \alpha>0$ and for every $\mu \in [0, a^*)$.
	\end{description}
	\begin{Theo}
		\label{posteriori-equation}
		Let Assumptions C and B1 holds. Assume that $0<\tau\sqrt{B_v}< \rho \sqrt{A_v}\|P_{\overline{rank\mathbf{K}}} y^\delta\|_Y $. 
		Then, there exists a constant $\alpha_D(\delta)$ such that
		the equation \eqref{eq-choosing-alp-posterior} holds. 
	\end{Theo}
	\begin{Rema}
		(i) For convenience of calculation, we can choose the parameter $\alpha$ such that $d(\alpha)\geq \tau' \sqrt{B_v}\delta$ for $\tau'>1$. Putting $\tau=\sqrt{B_v^{-1}}d(\alpha)$, we obtain the equation $d(\alpha)=\tau\sqrt{B_v}\delta$ and $\tau>\tau'>1$. \\
		(ii)
		Using the classical Morozov principle, we can choose
		the parameter $\alpha$ satisfying $\tilde{d}(\alpha)=\tilde{\tau}\delta$ for $\tilde{d}(\alpha)=\|Kx_\alpha^\delta-y^\delta\|_Y$ and $\tilde{\tau}$ is chosen appropriately. 
	\end{Rema}

	Next, we introduce some additional conditions suggested from \cite{Vainik}, page 75:
	
	\textbf{Assumption B2.} The function $g_\alpha:(0, a^*] \rightarrow \mathbb{R}$ satisfies
	\begin{description}
		\item[(i)] $g_\alpha(\mu) \geq 0$,
		\item[(ii)] $0 \leq 1-\mu g_\alpha(\mu) \leq \frac{g_\alpha(\mu)}{\ell_\alpha}$ with ${\ell_\alpha}:=\sup _{0 \leq \mu \leq a^*} g_\alpha(\mu)$,
		\item[(iii)] $\frac{\ell_*}{\alpha} \leq {\ell_\alpha} \leq \frac{\ell^*}{\alpha}$ with constants $\ell_*,\ell^*>0$.
	\end{description}
	In the following theorem, we provide a bound for $\Vert x_\alpha^\delta-x^{\ddagger}\Vert_X$. 
	
	We have the optimality of the rule of posteriori choice of the regularization parameter $\alpha_D(\delta)$ 
	\begin{Theo}
		\label{theo-posterior-general}
		Let $(\mathbf{u}, \mathbf{v}, \boldsymbol{\kappa})$ be a DFD  of $\mathbf{K}$ and 
		$\mathbf{u}$ be minimal.
		With $\mathbf{M}_{\varphi, E}$ defined as in (\ref{eq-source-1}) and Assumption A1 simultaneously satisfying that $x_\alpha^\delta$ is the approximate regularization solution as in (\ref{eq-solution-regular-noise-dfd}) with $g_\alpha$ satisfying Assumption B2. Moreover, assuming Assumption B1, (\ref{ieq-condition-C-delta}) is satisfied, and $\alpha_D(\delta)$ is chosen by the Morozov principle (\ref{eq-choosing-alp-posterior}). If the function $\varphi$ is concave, then 
		\begin{equation}
			\label{ieq-x-alp-del-x-poster-general}
			\Vert x_\alpha^\delta-x^{\ddagger} \Vert_X \leq \sqrt{A^{-1}_u B_u} (\tau+1) E \sqrt{ \Theta^{-1} \left(|\mathbf{v}|_{\inf}^2\delta^2/E^2\right)},
		\end{equation}
		where $A_u, B_u$ are bounds of frame $\mathbf{u} $.
		Hence,  $R_{\alpha_D(\delta)}$ is sequential order optimal over the set $\mathbf{M}_{\varphi, E}$. 
		
		Moreover, if there exists $\beta\in (0,1)$ and a $\delta_0>0$ such that $(0,\delta_0]\subset \cup_{\lambda\in\Lambda}D_{\lambda,\beta}$
		then the regularization method $R_{\alpha_D(\delta)}$ is uniform order optimal. Here we recall that $D_{\lambda,\beta}$ defined in \eqref{set-D}.
	\end{Theo}
	Note that our result will be as in Theorem 2.7 in \cite{Hubmer} if the source function $\varphi$ is replaced by polynomial function.

	\section{Illustrative problems}
	\subsection{Notations}
	In this section, we denote the inner product and the norm of $L^2(\mathbb{R})$ by $\langle.,.\rangle$ and $\|.\|$ respectively. The Fourier transform of a function $ f(t) $, $t\in\mathbb{R}$,   is defined as
	$ \mathcal{F}(\omega) = \int_{\mathbb{R}} f(t) e^{-i\omega .t} dt $
	and, the Mittag-Leffler function is understood as 
	$$E_{\alpha, \beta}(z) = \sum_{n=0}^{\infty} \frac{z^n}{\Gamma(\alpha n + \beta)},\quad  E_{1,1}(z)=e^z.$$
	Here the notation $\Gamma $ denotes the gamma function. From \cite{TrongNaneMinhTuan}, for $0<\gamma<1$, we have 
	\begin{equation}
		\dfrac{\underline{c}}{1+|z|} \leq \left| E_{\gamma, 1} \left( z \right) \right)| \leq \dfrac{\overline{c}}{1+|z|},~~~{\rm for}~ z<0,
		\label{bound-ML}
	\end{equation}
	where the derivative in the equation is understood in the Caputo sense. Specifically, The Caputo derivative is defined as:
	$$ \frac{{d^\gamma}}{{dt^\gamma}} f(t) =
	\begin{cases}
		\frac{1}{\Gamma(1-\alpha)} \int_{0}^{t} (t-\tau)^{-\gamma} \frac{d}{d\tau} f(\tau) d\tau, & 0<\gamma<1,\\
		\frac{df}{dt}, & \gamma=1.
	\end{cases}
	$$

	\subsection{Statement of the problems}
	We give an example to illustrate our result in previous section. In particular,  consider the backward fractional heat equation. Specifically, we find the solution at the initial time $u(x,0)$ knowing that $u(x,t)$ satisfies the heat equation. 
	Precisely, for $\gamma \in \left(0; 1 \right]$, we consider the equation
	\begin{equation}
		\label{eq-fheqs}
		\partial^\gamma_t u(x, t) - \Delta u(x, t) =0, x \in \mathbb{R}, 0 <t <T
	\end{equation}
	subject to the final condition 
	\begin{equation}
		\label{eq-finla-condition}
		u(x, T) = \theta_T (x),~~~x\in\mathbb{R}.
	\end{equation}	
	Similar to the condition \eqref{ieq-noise-org}, we have to consider the problem (\ref{eq-fheqs}) and (\ref{eq-finla-condition}) with the unknown exact data $\theta_T$  replaced by noise data $\theta_T^\delta$ satisfying 
	\begin{equation}
		\Vert \theta_T^\delta - \theta_T \Vert \leq \delta.
		\nonumber
	\end{equation}	
	We consider the ill-posedness of the problem. 
	Using the Fourier transform yields the solution to the problem yields
	\begin{equation}
		\label{eq-solution-fheqs}
		\mathcal{F}\theta_0 \left(\omega \right) = E^{-1}_{\gamma, 1} \left(- \left|\omega \right|^2 T^\gamma \right) \mathcal{F}\theta_T \left(\omega \right),
	\end{equation}
	where $\omega\in\mathbb{R}$. 
	From the inequality \eqref{bound-ML},  for $\gamma\in (0,1)$, the factor $E^{-1}_{\gamma, 1}(-|\omega|^2T^\gamma)$  has the polynomial magnitude $\asymp |\omega|^2$. This term causes instability when $|\omega|$ is large. Similarly, for $\gamma=1$, the factor $E^{-1}_{\gamma, 1}(-|\omega|^2T^\gamma)$  has the exponential magnitude $\asymp e^{|\omega|^2T}$ which causes severely instability when $|\omega|$ is large. 
	In conclusion, the problem is ``polynomially''  ill-posed for $\gamma\in (0,1)$ and is ``exponentially''  ill-posed for $\gamma=1$. These two types of ill-posed differ in nature. Therefore, suitable regularization methods for these two cases need to be proposed.

	We denote by $\mathbf{K}$ the operator such that 
	$\mathbf{K}\theta_0 = \theta_T$. Then, from the Fourier solution \eqref{eq-solution-fheqs} of the equation, we deduce that $\mathcal{F} \left( \mathbf{K}\theta_0 \right) = E_{\gamma, 1} \left(- \left|\omega \right|^2 T^\gamma \right) \mathcal{F}\theta_0 \left(\omega \right)= \mathcal{F} \left( \theta_T \right)$ . 	To find the DFD of $\mathbf{K}$, we base on the following arguments (see \cite{Ebner}). As we have known that, for all $\theta_0, v_0 \in \text{dom} \mathbf{K}$, we have 
	\begin{align*}
		\langle \mathbf{K}\theta_0, \mathbf{K} v_0 \rangle  &= 	
		\frac{1}{2\pi} \langle \mathcal{F} \mathbf{K}\theta_0, \mathcal{F} \mathbf{K} v_0 \rangle =  \frac{1}{2\pi} \langle  E_{\gamma, 1} \left(- \left|\omega \right|^2 T^\gamma \right) \mathcal{F}\theta_0, E_{\gamma, 1} \left(- \left|\omega \right|^2 T^\gamma \right) \mathcal{F} v_0 \rangle.
	\end{align*}
	It follows that $\mathbf{K}^* \mathbf{K} \theta_0 = \mathcal{F}^{-1} \left( \left| E_{\gamma, 1} \left(- \left|\omega \right|^2 T^\gamma \right)\right|^2  \mathcal{F}\theta_0 \right)$ for all $ \theta_0 \in \text{ dom }  \left(\mathbf{K}^*\mathbf{K}\right)  $. 
	
	To study DFD ($\mathbf{u}, \mathbf{v}, \boldsymbol{\kappa})$ of  operator $\mathbf{K}$  on $L^2\left(\mathbb{R}^2 \right)$. At first, we show necessary properties for $\mathbf{v}$ and $\boldsymbol{\kappa}$ in general case and after that we consider DFD for $\mathbf{K}$ in the case that $\mathbf{u}$ is concrete systems. 
	In fact, for each $v_\lambda \in \text{ran} \mathbf{K}$, because $\left(\mathbf{u}, \mathbf{v},\boldsymbol{ \kappa}\right)$ is a DFD for $\mathbf{K}$, we have $v_\lambda = \kappa_\lambda \mathbf{K} \sigma_\lambda$ for some $\sigma_\lambda \in \text{dom } K$. Moreover, from condition {\bf D3}, $\mathbf{K}^* v_\lambda = \kappa_\lambda u_\lambda$. It leads us to $u_\lambda = \mathbf{K}^* \mathbf{K} \sigma_\lambda$. This gives that $\sigma_\lambda \in \text{ dom }  \left(\mathbf{K}^*\mathbf{K}\right)$. Combining with the Fourier representation of $\mathbf{K}^*\mathbf{K}\sigma_\lambda$, we have $\sigma_\lambda =  \mathcal{F}^{-1} \left( \left| E_{\gamma, 1} \left(- \left|\omega \right|^2 T^\gamma \right)\right|^2 \mathcal{F} u_\lambda \right)$ and 
	\begin{equation}
		\label{eq-v-lambda}E_{\gamma, 1}
		v_\lambda = \kappa_\lambda \mathbf{K} \mathcal{F}^{-1} \left( \left|  \left(- \left|\omega \right|^2 T^\gamma \right)\right|^{-2} \mathcal{F}u_\lambda \right).
	\end{equation}
	From here, we can establish the relationship between the Fourier transform of $ v_\lambda$ and $u_\lambda $ as 
	\begin{equation}
		\label{eq-v-lambda-Fourier}
		\mathcal{F} v_\lambda = \kappa_\lambda \overline{E^{-1}_{\gamma, 1} \left(- \left|\omega \right|^2 T^\gamma \right)}  \mathcal{F}u_\lambda
		= \kappa_\lambda E^{-1}_{\gamma, 1} \left(- \left|\omega \right|^2 T^\gamma \right)  \mathcal{F}u_\lambda,
	\end{equation}
	where $\overline{z}$ is conjugate of $z$.
	
	\subsection{The fractional backward problem}
	We first consider the case $\gamma\in (0,1)$. 
	Denote the $L^2(\mathbb{R})$-wavelet orthonormal basis $\psi^{j,k}\left(x \right) = 2^{j/2} \psi \left(2^jx-k   \right), $ $(j,k)\in\mathbb{Z}^2$, where $\psi$ is a mother wavelet (see, e.g., \cite{Christ}, chap. 15). Put 
	$ \Lambda=\mathbb{Z}^2$, $\lambda=(\lambda_D,\lambda_T)\in\mathbb{Z}^2$. We consider the wavelet orthonormal basis ${u}_\lambda$ in $L^2(\mathbb{R}^2)$ in  the form
	\begin{equation}
		\label{eq-u-cuthe}
		{u}_\lambda(x)=\psi^{\lambda}(x),~~~\forall \lambda=(\lambda_D,\lambda_T) \in \Lambda, x\in\mathbb{R}.
	\end{equation}
	From here, we construct the DFD  for the operator $ \mathbf{K} $ using the following theorem.
	\begin{Theo}
		\label{Theo-WVD-K}
		Let $(u_{\lambda})_{\lambda \in \Lambda}$ be defined as in (\ref{eq-u-cuthe}) such that $supp \left(\mathcal{F}\psi\right) \subset \left\{ \omega \in \mathbb{R}^2: a_u \leq \left|\omega \right| \leq b_u \right\}$ where $a_u, b_u $ be positive constants. Then \\
	(a) $\left( u_\lambda, v_\lambda, \kappa_\lambda\right)_{\lambda \in \mathbb{Z}}$ be a DFD for $\mathbf{K}$ where 
		$$\kappa_\lambda = 
		\begin{cases}
			2^{-2\lambda_D}, & \text{for}~ \lambda_D\geq 1,\\
			1,     & \text{for}~   \lambda_D< 1.    
		\end{cases}
		$$
		and $v_\lambda$ defined as in (\ref{eq-v-lambda}).\\
	(b) If $\theta_0\in H^p(\mathbb{R})$ then $\theta_0\in \mathbf{M}_{\varphi,E}$, where $\varphi(\mu)=\mu^{p/2}$ and 
		$E$ large enough.
	(c) There exists a $\delta_0>0$ such that $(0,\delta_0]\subset
		\cup_{\lambda\in\Lambda}D_{\lambda,\beta}$ for 
		$\beta=2^{2+p}$. Here $D_{\lambda,\beta}$ is defined in 
		\eqref{set-D}.
	\end{Theo}
	\begin{Rema}
		(i) For polynomially ill-posed  problems, the WVD system can be used well. We can see that in the tomography problem (see \cite{Ebner, Hubmer}) and the fractional backward problem.\\
	(ii) The result (b) provides a sufficient condition for the function $\theta_0$ to satisfy the DFD source condition. The function only needs to lie in the Hilbert scales $H^p(\mathbb{R})$.  
	\end{Rema}
	
	From Theorem \ref{Theo-WVD-K}, we obtain the WVD of the operator $\mathbf{K}$. In particular, that is $\left(u_\lambda, v_\lambda, \kappa_{\lambda}\right)_{\lambda\in\Lambda}$.  
	This allows us to regularize the inverse problem for the fractional heat equation with the source function $\varphi\left(\mu \right) = \mu^{p/2}$ and then $\mathbf{M}_{\varphi, E}$ becomes 
	\begin{equation}
		\mathbf{M}_{\varphi, E} : = \left\{\theta_0 \in L^2\left(\mathbb{R} \right): \sum_{\lambda\in\Lambda}
		[\varphi(\kappa_\lambda^2)]^{-1}\langle \theta_0, u_\lambda \rangle|^2\leq E^2\right\}.
		\nonumber
	\end{equation}
	Using the Tikhonov regularization for $g_\alpha \left(\lambda \right)= \frac{1}{\alpha + \lambda}$. Then some conditions for filter function satisfy.
	The chosen $\{u_\lambda\}$ is tight, since it is 
	orthonormal. So  $\overline{u}_\lambda=u_\lambda $ and
	\eqref{eq-solution-regular-noise-dfd} can be rewritten as   
	$$ u_{0 \alpha}^\delta : = \sum_{\lambda \in \Lambda} \dfrac{2^{-2\lambda_D}}{\alpha + 2^{-4\lambda_D}} \langle \theta_T^\delta, v_\lambda \rangle u_\lambda. $$
	
	From the Theorem \ref{Theo:prior}, we have the result for the choice of a priori parameter and a posteriori parameter as the following consequence.
	\begin{Theo}
		\label{fractional-backward}
		For $(u_\lambda, v_\lambda, \kappa_\lambda)_{\lambda\in\Lambda}$ being the constructed WVD 
		of $\mathbf{K}$ and $\theta_0 \in \mathbf{M}_{p, E}$ for $0<p\leq 4$. 
		
		(a) (apriori regularizarion) For $0<p\leq 4$, if we choose the regularization parameter as
		\begin{equation*}
			\label{eq-parameter-choose-prior-1}
			\alpha=\left(\delta / E\right)^\frac{2}{p+2},
		\end{equation*}
		then $u_{0\delta}^\delta$ is uniform order optimal and
		the following convergence rate result holds
		\begin{equation*}
			\label{ieq-prior-1}
			\left\|u_{0 \alpha}^\delta -  \theta_0\right\| \leq  C \delta^\frac{p}{p+2} E^\frac{2}{p+2},
		\end{equation*}
		where $B_v$ are bound of  $v_j$ as in Theorem (\ref{Theo-WVD-K}). 
		
		(b) (posteriori regularization) If $0<p\leq 2$, assume that  $\alpha_D$ is chosen by the Morozov principle (\ref{eq-choosing-alp-posterior}). Then $u_{0\alpha_D}^\delta$ is uniform order optimal over the set $\mathbf{M}_{\varphi, E}$, and
		\begin{equation}
			\label{ieq-x-alp-del-x-poster-general-1}
			\Vert u_{0\alpha_D}^\delta-\theta_0 \Vert \leq   C \delta^\frac{p}{p+2} E^\frac{2}{p+2}.\nonumber
		\end{equation}
	\end{Theo}
	
	\begin{Rema}
		If $p=4\nu$, we obtain the error stated in \cite{Ebner}.
	\end{Rema}
	\subsection{The backward problem of the heat equation}
	Put $B_N=\{\omega\in\mathbb{R}: \sqrt{N}\leq |\omega|\leq \sqrt{N+1}\} $, $N\in\mathbb{N}$.
	From here, we construct the DFD  for the operator $ K $ using the following theorem.
	\begin{Theo}
		\label{Theo-WVD-K1}
		Let $(u_{\lambda})_{\lambda \in \Lambda}$ defined as in (\ref{eq-u-cuthe}) such that $supp \left(\mathcal{F}\psi\right) \subset \left\{ \omega \in \mathbb{R}: a_u \leq \left|\omega \right| \leq b_u \right\}$ where $a_u, b_u $ be positive constants and where $\lambda=(\lambda_D,\lambda_T)$. Put 
		$u_{\lambda,N}=\mathcal{F}^{-1}(\mathbf{1}_{B_N}
		\mathcal{F}(u_\lambda))$.	
		Then 
		
		(a) $\left(u_{\lambda,N},v_{\lambda,N},\kappa_{\lambda,N}\right)_{\lambda \in \mathbb{Z}, N\in\mathbb{N}}$ be a DFD for $\mathbf{K}$ where 
		$$\kappa_{\lambda,N} = e^{-NT}, ~~~ N\in\mathbb{N}
		$$
		and $v_\lambda=\kappa_{\lambda,N}\mathcal{F}(e^{|\omega|^2T^\gamma}u_{\lambda,N})$ defined as in (\ref{eq-v-lambda}). Moreover, $u_{\lambda,N}$ is a tight frame.
		
		(b) If $\theta_0\in H^p(\mathbb{R})$ then $\theta_0\in \mathbf{M}_{\varphi,E}$, where $\varphi(\mu)=(-\ln \mu)^{-p}$ and 
		$E$ large enough.
		
		(c)  There exists a $\delta_0>0$ such that $(0,\delta_0]\subset
		\cup_{(\lambda,N)\in\Lambda\times\mathbb{N}}D_{(\lambda,N),\beta}$ for 
		$\beta=e^{-T}$. Here we recall $D_{(\lambda,N),\beta}$ is defined in \eqref{set-D}.
	\end{Theo}
	\begin{Rema}
		Using the classical wavelet system as in the previous section, we cannot find a suitable $\kappa_\lambda$. Therefore, it is necessary to construct a DFD system. There are many ways to construct the system mentioned. However, we use a system that inherits the classical wavelet system as presented.
	\end{Rema}
	From Theorem \ref{Theo-WVD-K1}, we obtain the WVD of the operator $\mathbf{K}$. In particular, that is $\left(u_{\lambda,N}, v_{\lambda,N}, \kappa_{\lambda,N}\right)_{\lambda\in\Lambda, N\in\mathbb{N}}$.  
	This allows us to regularize the inverse problem for the fractional heat equation with the source function $\varphi\left(\lambda \right) = (-\ln \lambda)^{-p}$.
	Using the Tikhonov regularization for $g_\alpha \left(\mu \right)= \frac{1}{\alpha + \mu}$.  Then the approximate solution can be written in the form
	\eqref{eq-solution-regular-noise-dfd}.
	The chosen $\{u_\lambda\}$ is a tight frame, so  $\overline{u}_\lambda=u_\lambda $ and
	\eqref{eq-solution-regular-noise-dfd} can be rewritten as   
	$$ u_{0 \alpha}^\delta : = \sum_{N\in\mathbb{N}}\sum_{\lambda \in \Lambda} 
	\dfrac{\kappa_{\lambda,N}}{\alpha + \kappa_{\lambda,N}^2} \langle \theta_T^\delta, v_{\lambda,N} \rangle u_{\lambda,N}. $$
	
	From the Theorem \ref{Theo:prior}, we have the result for the choice of a priori parameter as the following consequence.
	Clearly, the Assumption A1, A2, B1, B2 being satisfied, from Theorem \ref{Theo:lower-bound} and Theorem \ref{theo-posterior-general}, we deduce the following consequence.
	
	\begin{Theo}
		\label{classical-backward}
		For $(u_{\lambda,N}, v_{\lambda,N}, \kappa_{\lambda,N})_{\lambda\in\Lambda, N\in\mathbb{N}}$ being the constructed WVD 
		of $\mathbf{K}$ and $\theta_0 \in \mathbf{M}_{\varphi, E}$ where $\varphi(\mu)=(-\ln\mu)^p$.  
		If we choose the regularization parameter
		$\alpha=\delta/E$
		then $u_{0\delta}^\delta$ is uniform order optimal and
		the following convergence rate result holds
		\begin{equation}
			\label{ieq-prior-1}
			\left\|u_{0 \alpha}^\delta -  \theta_0\right\| \leq  CE \left(-\ln(\delta/E)\right)^{-p}.
			\nonumber
		\end{equation}
		
	\end{Theo}
	\begin{Rema}
		The system $\{ u_{\lambda,N}\}$ is unlikely to satisfy the minimal property. Therefore the optimal results in the theorems \ref{Theo:lower-bound}, \ref{Theo:prior},
		\ref{theo-posterior-general} cannot be applied.
		However, based on theorem \ref{Theo-WVD-K1}, (b), we can obtain the optimal result for the a priori case. 
	\end{Rema}

	\section{Proofs}
	\subsection{Preliminary lemmas}
	\begin{Lemm}
		Let $(\mathbf{u},\mathbf{v}, \boldsymbol{\kappa})$ be the DFD. Assume in addition that $\mathbf{u}$ is a Riesz basis. Then we have
		\label{u-to-v}
		\begin{align*}
			\langle x^\delta_\alpha,u_\lambda \rangle_X&=\kappa_\lambda 
			g_\alpha(\kappa_\lambda^2)\langle y^\delta,v_\lambda\rangle_Y, ~
			\langle x^\ddagger,u_\lambda \rangle_X=\frac{1}{\kappa_\lambda}\langle y,v_\lambda\rangle_Y,\\
			\langle \mathbf{K}x_\alpha^\delta,v_\lambda\rangle_Y &=	
			\kappa_\lambda^2g_\alpha(\kappa^2_\lambda)\langle y^\delta,v_\lambda\rangle_Y,\\
			\langle \mathbf{K}x_\alpha,v_\lambda\rangle_Y &=
			\kappa_\lambda^2g_\alpha(\kappa^2_\lambda)\langle y,v_\lambda\rangle_Y,\\
			\langle \mathbf{K}x^\ddagger,v_\lambda\rangle_Y &=\langle y,v_\lambda\rangle_Y.
		\end{align*}
		
	\end{Lemm}
	\begin{proof}
		Since $\langle \mathbf{u}_\lambda,\overline{\mathbf{u}}_\mu\rangle=\delta_{\lambda\mu} $,	we have
		$$
		\langle x^\delta_\alpha,u_\lambda \rangle_X=\kappa_\lambda 
		g_\alpha(\kappa_\lambda^2)\langle y^\delta,v_\lambda\rangle_Y, ~
		\langle x^\ddagger,u_\lambda \rangle_X=\frac{1}{\kappa_\lambda}\langle y,v_\lambda\rangle_Y
		$$
		From \eqref{eq-solution-regular-noise-dfd}, \eqref{eq-solution-regular-dfd},\eqref{eq-present-x} we have
		\begin{align*}
			\langle \mathbf{K}x_\alpha^\delta,v_\lambda\rangle_Y &=	
			\langle x_\alpha^\delta,\mathbf{K}^*v_\lambda \rangle_X
			=\kappa_\lambda \langle x_\alpha^\delta,u_\lambda \rangle_X=
			\kappa_\lambda^2g_\alpha(\kappa^2_\lambda)\langle y^\delta,v_\lambda\rangle_Y,\\
			\langle \mathbf{K}x_\alpha,v_\lambda\rangle_Y &=	
			\langle x_\alpha,\mathbf{K}^*v_\lambda \rangle_X
			=\kappa_\lambda \langle x_\alpha,u_\lambda \rangle_X=
			\kappa_\lambda^2g_\alpha(\kappa^2_\lambda)\langle y,v_\lambda\rangle_Y,\\
			\langle \mathbf{K}x^\ddagger,v_\lambda\rangle_Y &=\langle y,v_\lambda\rangle_Y.
		\end{align*}
	\end{proof}
	\begin{Lemm}
		\label{Theta}
		Let the function $\varphi$ satisfies Assumption A1 and $\Theta(\mu)=\mu\varphi^{-1}(\mu)$.
		If $\varphi$ is concave on $(0,a^*)$ then 
		$\Theta^{-1}(t^2z)\geq t\Theta^{-1}(z)$ for every $t\in (0,1)$, $z\in (0,a^*\varphi(a^*))$.
	\end{Lemm}
	\begin{proof}
		From the definition of the function $\Theta$, we have $\Theta^{-1}(\mu\varphi(\mu))=\varphi(\mu)$. From the concave of the function  $\varphi$ and  $\lim _{\mu \rightarrow 0} \varphi(\mu)=0$. It leads us to  $t \varphi(\mu) \leq \varphi(t \mu)$ for  $t \in[0,1]$. Equivalently,  $\varphi^{-1}(t \varphi(\mu)) \leq \lambda t$.  We also have,  $\Theta(\mu):=\mu \varphi^{-1}(\mu)$, it follows that  $\Theta(t \varphi(\mu)) \leq t^2 \mu \varphi(\mu)$. Hence, $\Theta^{-1}(t^2\mu\varphi(\mu))\geq 
		t\varphi(\mu)=\Theta^{-1}(\mu\varphi(\mu))$. Putting $z=\mu\varphi(\mu)$, we obtain 
		the desired inequality.

	\end{proof}
	\subsection{Proof of Theorem \ref{well-posed}}
	Using \eqref{w-dual-upper-inequality}, we obtain
	\begin{align*}
		\|\mathbf{K}^{\ddagger}y^\delta-x^{\ddagger}\|_X
		&=\left\| \sum_{\lambda\in\Lambda}\frac{1}{\kappa_\lambda}
		\langle y^\delta-y,v_\lambda\rangle_Y\overline{u}_\lambda\right\| \leq \frac{1}{\kappa_0|\mathbf{u}|_{\inf}} 
		\left( \sum_{\lambda\in\Lambda}|\langle y^\delta-y,v_\lambda\rangle_Y|^2\right)^{1/2}.  
	\end{align*}
	Combining the latter inequality with \eqref{Bw-general} yields
	$$\|\mathbf{K}^{\ddagger}y^\delta-x^{\ddagger}\|_X
	\leq \frac{|\mathbf{v}|_{\sup}}{\kappa_0|\mathbf{u}|_{\inf}}\|y^\delta-y\|_Y\leq \frac{|\mathbf{v}|_{\sup}\delta}{\kappa_0|\mathbf{u}|_{\inf}}.  $$
	This completes the proof of Theorem \ref{well-posed}.
	\subsection{Proof of Theorem \ref{convergence}  }
	The proof of Theorem \ref{convergence} could be seen in \cite{Ebner} or in the proof of Theorem 
	\ref{Theo:prior} below.
	
	\subsection{Proof of Theorem \ref{Theo:lower-bound}.}
	\begin{proof}
		One of the commonly used methods to find a lower bound for the worst-case error  is to compute the modulus of continuity  
		$$\Omega\left(\mathbf{M}, \delta \right) = \sup \{\|x\| \mid x \in \mathbf{M} \wedge\|\mathbf{K} x\| \leq \delta\} .$$
		As is known (see, e.g. \cite{Tautenhahn}), we have
		$$ \Delta(\mathbf{M}_{\varphi, E},\delta,\mathbf{R})\geq \Omega\left(\mathbf{M}_{\varphi, E}, \delta \right).  $$ 
		For any 
		$\nu \in \Lambda$ set $x_\nu:= E \sqrt{\varphi \left(\kappa^2_\nu \right)} \overline{u}_\nu $ such that
		$$
		\left\langle x_\nu, u_\lambda\right\rangle= \sqrt{\varphi \left(\kappa^2_\nu \right)} \omega_\lambda, \quad \omega_\lambda= 
		\begin{cases}
			E, & \text { if } \lambda=\nu, \\ 0, & \text { else. }
		\end{cases}
		$$
		By definition we have $\|\omega\|_2=E$ and $x_\nu \in \mathbf{M}_{\varphi, E}$. Choosing $\nu\in D_{\lambda,\beta}$, we can write  $\delta^*_\nu=|\mathbf{v}|_{\inf}^{-1}E  \sqrt{\kappa^2_\nu \varphi \left(\kappa^2_\nu \right)}\leq \delta $ and obtain
		\begin{equation}
			\left\|x_\nu\right\|^2 \geq \frac{1}{B_u} \sum_{\lambda \in \Lambda}\left|\left\langle u_\lambda, x_\nu\right\rangle_X\right|^2=\frac{1}{B_u}  \varphi \left(\kappa^2_\nu \right) E^2.
			\label{ieq-x-modulus}
		\end{equation}
		Moreover, combine   the Asumption A1 (ii) and the choosing of $\delta_\nu$, then $\Theta\left(\varphi\left(\kappa^2_\nu \right)\right) = \kappa^2_\nu \varphi \left(\kappa^2_\nu \right)=
		|\mathbf{v}|_{\inf}^2 \delta_\nu^{*2}/E^2$. From here, it leads us to 
		$\varphi(\kappa_\nu^2)=
		\Theta^{-1}(|\mathbf{v}|_{\inf}^2
		\delta_\nu^{*2}/E^2). $
		After that, taking it to \eqref{ieq-x-modulus}, we obtain 
		\begin{equation}
			\left\|x_\nu\right\| \geq \sqrt{B_u^{-1}} E \sqrt{\Theta^{-1}\left( |\mathbf{v}|_{\inf}^2\delta_\nu^{*2} / E^2  \right)}. 
			\label{ieq-x-modulus1}
		\end{equation}
		On the other hand, we also have 
		$$
		\left\|\mathbf{K} x_\nu\right\|^2 \leq \frac{1}{|\mathbf{v}|_{\inf}^2} \sum_{\lambda \in \Lambda}\left|\left\langle \mathbf{K} x_\nu , v_\lambda\right\rangle\right|^2=
		\frac{1}{|\mathbf{v}|_{\inf}^2} \sum_{\lambda \in \Lambda} \kappa^2_\lambda\left|\left\langle x_\nu, u_\lambda \right\rangle\right|^2=\frac{1}{|\mathbf{v}|_{\inf}^2} \kappa^2_\nu \varphi\left( \kappa^2_\nu \right) E^2=\delta_\nu^{*2}\leq \delta^2 .
		$$
		Hence $\left\|\mathbf{K} x_\nu\right\| \leq \delta$ and $x_\nu \in \mathbf{M}_{\varphi, E}$. From the definition of $\Omega \left(\mathbf{M}_{\varphi, E}, \delta_\nu \right)$ and (\ref{ieq-x-modulus1}), we deduce that 
		$$  \Omega \left(\mathbf{M}_{\varphi, E}, \delta \right) \geq \sqrt{B_u^{-1}} E \sqrt{\Theta^{-1}\left(|\mathbf{v}|_{\inf}^2 \delta^{*2}_\nu / E^2  \right)}. $$
		Using Lemma \ref{Theta} gives 
		$$\Theta^{-1}\left(|\mathbf{v}|_{\inf}^2 \delta_\nu^{*2} / E^2  \right)=
		\Theta^{-1}\left(\beta^2|\mathbf{v}|_{\inf}^2 \beta^{-2}\delta_\lambda^{*2} / E^2  \right) 
		\geq \Theta^{-1}\left(\beta^2|\mathbf{v}|_{\inf}^2\delta^2 / E^2  \right)
		\geq \beta \Theta^{-1}\left(|\mathbf{v}|_{\inf}^2\delta^2 / E^2  \right).$$
	\end{proof}

	\subsection{Proof of  Theorem \ref{Theo:prior}.}
	
	\begin{proof}
		The regularization solution for the equation (\ref{eq-operator-org}) for the noiseless  has  the form 
		\begin{equation}
			\label{eq-solution-regular-dfd}
			x^\alpha: = R_\alpha y = \sum_{\lambda \in \Lambda} \kappa_\lambda g_\alpha \left(\kappa_\lambda^2 \right) \langle y, v_\lambda \rangle \overline{u}_\lambda,
		\end{equation}
		From the triangle inequality, it gives us
		\begin{equation}
			\label{ieq-triangle-org}
			\Vert x^\delta_\alpha -  x^{\ddagger} \Vert_X  \leq \Vert x^\delta_\alpha -  x_\alpha \Vert_X + \Vert x_\alpha -  x^{\ddagger} \Vert_X. 
		\end{equation}
		For the first term on the right hand side, using (\ref{eq-solution-regular-dfd}) and  (\ref{eq-solution-regular-noise-dfd}) gives 
		\begin{align}
			\Vert x^\delta_\alpha -  x_\alpha \Vert_X  & =  \left\| \sum_{\lambda \in \Lambda} \kappa_\lambda g_\alpha \left(\kappa_\lambda^2 \right) \langle y^\delta-y, v_\lambda \rangle_Y \overline{u}_\lambda \right\|_X. \nonumber
		\end{align}
		From \eqref{AuBu} and \eqref{Bw-general}, the inequality \eqref{w-dual-upper-inequality} yields  
		\begin{align}
			\Vert x^\delta_\alpha -  x_\alpha \Vert_X 
			& \leq  \dfrac{1}{\sqrt{A_u}} \left( \sum_{\lambda \in \Lambda} \sup_\lambda \left(\kappa^2_\lambda g^2_\alpha \left(\kappa_\lambda^2 \right)\right) \left| \langle y^\delta-y, v_\lambda \rangle_Y \right|^2 \right)^\frac{1}{2} \nonumber \\
			& \leq   \dfrac{1}{\sqrt{A_u}} \dfrac{\gamma_1}{\sqrt{\alpha}} \left( \sum_{\lambda \in \Lambda}  \left| \langle y^\delta-y, v_\lambda \rangle_Y \right|^2 \right)^\frac{1}{2} \nonumber \\
			& \leq  \sqrt{\frac{B_v}{A_u}} \dfrac{\gamma_1}{\sqrt{\alpha}}  \Vert y^\delta - y \Vert_Y  
			\leq  \gamma_1 \sqrt{B_v A_u^{-1}} \dfrac{\delta}{\sqrt{\alpha}}.
			\label{x-del-alp-x-alp}
		\end{align}
		The second line is a result of utilizing Assumption A2 (i). The penultimate line is obtained through the upper bound of the frame $\left( v_\lambda \right)_{\lambda \in \Lambda}$. Finally, the last line is due to the noise condition (\ref{ieq-noise-org}).\\		
		For the last term in (\ref{ieq-triangle-org}), using (\ref{eq-solution-dfd}) and  (\ref{eq-solution-regular-dfd}), it follows that 
		\begin{align}
			\label{x-alp-x-dagger}
			\Vert x_\alpha - x^\ddagger \Vert_X  & =  \left\| \sum_{\lambda \in \Lambda} \left( \kappa^2_\lambda g_\alpha \left(\kappa_\lambda^2 \right) -1 \right)  \langle x^\ddagger, u_\lambda \rangle_X \overline{u}_\lambda \right\|_X \nonumber \\ 
			& \leq  \dfrac{1}{\sqrt{A_u}} \left( \sum_{\lambda \in \Lambda} \left|1- \kappa^2_\lambda g_\alpha \left(\kappa_\lambda^2 \right)\right|^2 \left| \langle x^\ddagger, u_\lambda \rangle_X \right|^2 \right)^\frac{1}{2} \nonumber \\
			& \leq  \dfrac{1}{\sqrt{A_u}} \left( \sum_{\lambda \in \Lambda} \left|1- \kappa^2_\lambda g_\alpha \left(\kappa_\lambda^2 \right)\right|^2 \left| \sqrt{\varphi\left( \kappa^2_\lambda \right) }\omega_\lambda \right|^2 \right)^\frac{1}{2} \nonumber \\
			& \leq  \dfrac{\gamma_2}{\sqrt{A_u}} \sqrt{\varphi\left(\alpha \right)} \left( \sum_{\lambda \in \Lambda}  \left| \omega_\lambda \right|^2 \right)^\frac{1}{2}  \leq  \dfrac{\gamma_2}{\sqrt{A_u}} \sqrt{\varphi\left(\alpha \right)} E.
		\end{align}
		The second and the last lines are derived from the exact solution belonging to the source set $ x^\ddagger \in \mathbf{M}_{\varphi, E} $. As for the third line, it is obtained from Assumption A2 (ii).\\	
		Combining (\ref{x-del-alp-x-alp}) and (\ref{x-alp-x-dagger}) yields 
		\begin{align*}
			\Vert x_\alpha^\delta - x^\ddagger \Vert_X \leq \gamma_1 \sqrt{B_v A_u^{-1}} \dfrac{\delta}{\sqrt{\alpha}} +  \dfrac{\gamma_2}{\sqrt{A_u}} \sqrt{\varphi\left(\alpha \right)} E.
		\end{align*}
		From the choice of the parameter (\ref{eq-parameter-choose-prior}), the regularized parameter is chosen by $\alpha=\alpha^*\left(\delta, y^\delta\right)=\varphi^{-1}\circ \Theta^{-1}\left(A_v\delta^2 / E^2\right)$ it leads us to  $ \varphi \left(\alpha \right) = \Theta^{-1}\left(A_v\delta^2 / E^2\right)$. From here, it is easy to see that  $\Theta \left( \varphi \left(\alpha \right) \right) =A_v \delta^2/E^2$ and it leads us to  $\delta^2 = A_v^{-1}E^2 \Theta \left( \varphi \left(\alpha \right) \right)$. Combine with Assumption A1 (iii), we obtain that
		\begin{equation*}
			\label{eq-choose-alp-prior-general-1}
			\dfrac{\delta}{\sqrt{\alpha}} = \sqrt{\dfrac{\delta^2}{\alpha}}=\sqrt{ \dfrac{A_v^{-1}E^2 \Theta \left( \varphi \left(\alpha \right) \right)}{\alpha}} = \sqrt{A_v^{-1}E^2 \varphi \left(\alpha \right)} = E \sqrt{A_v^{-1}\Theta^{-1}\left(A_v\delta^2 / E^2\right)}.
		\end{equation*}
		Hence, we get that
		\begin{align}
			\label{ieq-x-alp-del-x-dagger}
			\Vert x_\alpha^\delta - x^\ddagger \Vert_X & \leq   \gamma_1\sqrt{A_u^{-1}}\sqrt{A_v^{-1}} \sqrt{B_v } E \sqrt{\Theta^{-1}\left(A_v\delta^2 / E^2\right)} + \dfrac{\gamma_2}{\sqrt{A_u}} E \sqrt{\Theta^{-1}\left(A_v\delta^2 / E^2\right)} \nonumber\\
			& = \sqrt{A_u^{-1}}\sqrt{A_v^{-1}} \left(\gamma_1 \sqrt{B_v} +  \gamma_2\sqrt{A_v}  \right)E \sqrt{\Theta^{-1}\left(A_v\delta^2 / E^2\right)}\nonumber\\
			& = \sqrt{A_u^{-1}}\sqrt{A_v^{-1}} \left(\gamma_1 \sqrt{B_v} +  \gamma_2\sqrt{A_v}  \right)E \sqrt{\Theta^{-1}\left(|\mathbf{v}|_{\inf}^2\delta^2 / E^2\right)} 
			.\nonumber
		\end{align}
		The above estimate completes the proof.
	\end{proof}
	\subsection{Proof of Theorem \ref{posteriori-equation}}
	
	Using Lemma \ref{u-to-v}, we obtain 
	$$ d(\alpha)=\left(\sum_\lambda (1-\kappa_\lambda^2g_\alpha(\kappa_\lambda^2)
	)^2|\langle y^\delta,v_\lambda\rangle_Y|^2\right)^{\frac{1}{2}}. $$
	We know that under the conditions (i) - (ii) of Assumption B1, the function $d$ is continuous and has the following results:
	\begin{equation}
		\label{eq-limit-d}
		\lim _{\alpha \rightarrow 0} d(\alpha)=0 \text { and } \quad \lim _{\alpha \rightarrow \infty} d(\alpha)= \rho\left( \sum_{\lambda} \left| \langle y^\delta, v_\lambda \rangle \right|^2 \right)^\frac{1}{2}.
	\end{equation}
	We have
	\begin{equation}
		\label{ieq-condition-C-delta}
		0 <\tau\sqrt{B_v} \delta < \rho\sqrt{A_v} \Vert P_{\overline{\text{ran} \mathbf{K} }} y^\delta \Vert\leq \rho\left( \sum_{\lambda} \left| \langle y^\delta, v_\lambda \rangle \right|^2 \right)^\frac{1}{2} .
	\end{equation}
	Therefore, under Assumption B1 and \eqref{eq-limit-d}, the equation (\ref{eq-choosing-alp-posterior}) has a solution $\alpha=\alpha_D(\delta).$

	\subsection{Proof of Theorem \ref{theo-posterior-general}}
	First, with Assumption B2, using the ideas in \cite{Vainik}, page 77, we obtain the following inequality
	\begin{Lemm}
		For $x\in X$, we have
		\begin{align}
			\|x^\delta_\alpha-x^\ddagger\|_u^2+\ell_\alpha (\|\mathbf{K}x_\alpha^\delta-y^\delta\|_v^2-
			\|\mathbf{K}x^\ddagger-y^\delta\|_v^2)\leq 
			\langle \left[I-\mathbf{K}^* \mathbf{K} g_\alpha \left(\mathbf{K}^*\mathbf{K}\right)\right] x^{\ddagger}, x^{\ddagger}\rangle_u.
			\label{ieq-support-posterior}	
		\end{align}
		\label{lemm:support-posterior}
	\end{Lemm}
	\textbf{Proof of Lemma \ref{lemm:support-posterior}.}
	\begin{proof}
		We claim that
		\begin{align}
			|\langle x^\delta_\alpha-x^\ddagger,u_\lambda \rangle_X|^2
			+\ell_\alpha |\langle \mathbf{K}x_\alpha^\delta-y^\delta,v_\lambda \rangle_Y|^2-
			\ell_\alpha |\langle \mathbf{K}x^\ddagger-y^\delta,v_\lambda \rangle_Y|^2
			\leq (1-\kappa_\lambda^2g_\alpha(\kappa_\lambda^2))|
			\langle x^\ddagger,u_\lambda\rangle_X|^2.
			\label{ieq-scalar-posterior}
		\end{align}
		Using \eqref{u-to-v}, we can write
		\begin{align*}
			&|\langle x^\delta_\alpha-x^\ddagger,u_\lambda \rangle_X|^2
			=\frac{1}{\kappa_\lambda^2}\left|
			\kappa_\lambda^2 
			g_\alpha(\kappa_\lambda^2)\langle y^\delta-y,v_\lambda\rangle_Y-
			r_\alpha(\kappa_\lambda^2)
			\langle y,v_\lambda\rangle_Y   \right|^2\\
			&= \kappa_\lambda^2 
			g_\alpha^2(\kappa_\lambda^2)|\langle y^\delta-y,v_\lambda\rangle_Y|^2-2g_\alpha(\kappa_\lambda^2)
			r_\alpha(\kappa_\lambda^2)\langle y^\delta-y,v_\lambda\rangle_Y
			\langle y,v_\lambda\rangle_Y+\frac{r_\alpha^2(\kappa_\lambda^2)}{\kappa_\lambda^2}|\langle y,v_\lambda\rangle_Y|^2.			
		\end{align*}
		On the other hand, we have
		\begin{align*}
			&	g_\alpha(\kappa_\lambda^2)r_\alpha(\kappa_\lambda^2)
			|\langle y^\delta,v_\lambda\rangle_Y|^2=
			g_\alpha(\kappa_\lambda^2)r_\alpha(\kappa_\lambda^2)
			|\langle y^\delta-y,v_\lambda\rangle_Y+\langle y,v_\lambda\rangle_Y|^2\\
			&=g_\alpha(\kappa_\lambda^2)r_\alpha(\kappa_\lambda^2)
			(|\langle y^\delta-y,v_\lambda\rangle_Y|^2+
			2\langle y^\delta-y,v_\lambda\rangle_Y\langle y,v_\lambda\rangle_Y
			+|\langle y,v_\lambda\rangle_Y|^2 ).
		\end{align*}	 
		Adding two inequalities yield
		\begin{align*}
			&|\langle x^\delta_\alpha-x^\ddagger,u_\lambda \rangle_X|^2
			+	g_\alpha(\kappa_\lambda^2)r_\alpha(\kappa_\lambda^2)
			|\langle y^\delta,v_\lambda\rangle_Y|^2
			=g_\alpha(\kappa_\lambda^2)
			|\langle y^\delta-y,v_\lambda\rangle_Y|^2+
			\frac{r_\alpha(\kappa_\lambda^2)}{\kappa_\lambda^2}|\langle y,v_\lambda\rangle_Y|^2.	
		\end{align*}
		Since $\ell_\alpha\geq g_\alpha(\mu)\geq \ell_\alpha r_\alpha(\mu)$ we obtain
		\begin{align*}
			g_\alpha(\kappa_\lambda^2)r_\alpha(\kappa_\lambda^2)
			|\langle y^\delta,v_\lambda\rangle_Y|^2 &\geq \ell_\alpha 
			|r_\alpha(\kappa_\lambda^2)\langle y^\delta,v_\lambda\rangle_Y|^2=\ell_\alpha|\langle \mathbf{K}x^\delta_\alpha-y^\delta,v_\lambda \rangle|^2,\\
			g_\alpha(\kappa_\lambda^2)
			|\langle y^\delta-y,v_\lambda\rangle_Y|^2 &\leq \ell_\alpha
			|\langle y^\delta-y,v_\lambda\rangle_Y|^2=\ell_\alpha
			\langle \mathbf{K}x^\ddagger-y^\delta,v_\lambda \rangle|^2.
		\end{align*}
		Hence the inequality \eqref{ieq-scalar-posterior} holds. 
		Taking the sum of the latter inequlities with respect to 
		$\lambda\in\Lambda$, we get \eqref{ieq-support-posterior}.
	\end{proof}
	\textbf{Proof of Theorem \ref{theo-posterior-general}.}
	
	\begin{proof}
		Firstly, for every $(\alpha, \lambda)$, recalling the definition of $g_\alpha$, we reiterate that $r_\alpha \left(\lambda\right)$ is defined as follows:
		$$
		r_\alpha(\lambda)=1-\lambda g_\alpha(\lambda) .
		$$
		Let $\alpha=\alpha_D$ be the regularization parameter chosen by (\ref{eq-choosing-alp-posterior}). 
		Let $\alpha=\alpha_D$ be the regularization parameter chosen by (\ref{eq-choosing-alp-posterior}). Using \eqref{Bw-general} and (\ref{ieq-noise-org}) yields 
		$$\Vert \mathbf{K} x_\alpha^{\delta}-y^\delta\Vert_v \geq \tau\sqrt{B_v}\delta>\tau\sqrt{B_v}  \Vert \mathbf{K} x^{\ddagger}-y^\delta \Vert_Y\geq \Vert \mathbf{K} x^{\ddagger}-y^\delta \Vert_v . $$
		We recall that 
		$\Vert \mathbf{K} x_\alpha^{\delta}-y^\delta\Vert_v \geq \Vert \mathbf{K} x^{\ddagger}-y^\delta \Vert_Y\geq \Vert \mathbf{K} x^{\ddagger}-y^\delta \Vert_v . $
		It follows that 
		\begin{equation}
			\Vert x_\alpha^\delta-x^{\ddagger}\Vert_u \leq \langle \left[I-\mathbf{K}^* \mathbf{K} g_\alpha \left(\mathbf{K}^* \mathbf{K}\right)\right] x^{\ddagger}, x^{\ddagger}\rangle^{\frac{1}{2}}_u = \Vert \left[r_\alpha\left(\mathbf{K}^* \mathbf{K}\right)\right]^{\frac{1}{2}} x^{\ddagger}\Vert_u.
			\label{ieq-norm-inner-produc}
		\end{equation}
		Applying Lemma \ref{Theta} for   $t=$ $r_\alpha(\lambda):=1-\lambda g_\alpha(\lambda)$  give 
		\begin{equation}
			\label{ieq-varrho-r}
			\Theta \left(r_\alpha(\lambda) \varphi(\lambda)\right) \leq \lambda r_\alpha^2(\lambda) \varphi(\lambda) .
		\end{equation}
		For  $x_\alpha$ be a approximate solution as in  (\ref{eq-solution-regular-dfd}) in the case  $y^\delta$ which is replaced by exact data  $y$. 
		
		Using Lemma \ref{u-to-v} yields
		\begin{align*}
			\Vert \mathbf{K} x_\alpha-\mathbf{K} x^{\ddagger} \Vert_v & = 
			\left(\sum_{\lambda\in\Lambda}|\langle \mathbf{K}x_\alpha-\mathbf{K} x^{\ddagger},v_\lambda\rangle_Y 
			|^2\right)^{1/2}= \left(\sum_{\lambda\in\Lambda}
			r_\alpha(\kappa_\lambda^2)|\langle y,v_\lambda\rangle |^2\right)^{1/2}\\
			&\leq \left(\sum_{\lambda\in\Lambda}
			r_\alpha(\kappa_\lambda^2)|\langle y^\delta,v_\lambda\rangle |^2\right)^{1/2}+
			\left(\sum_{\lambda\in\Lambda}
			r_\alpha(\kappa_\lambda^2)|\langle y-y^\delta,v_\lambda\rangle |^2\right)^{1/2}\\
			&\leq  
			\left(\sum_{\lambda\in\Lambda}|\langle \mathbf{K}x_\alpha^\delta-y^\delta,v_\lambda\rangle_Y 
			|^2\right)^{1/2}+
			\left(\sum_{\lambda\in\Lambda}
			|\langle y-y^\delta,v_\lambda\rangle |^2\right)^{1/2}. 
		\end{align*}
		Because  $\alpha=\alpha_D$ be a solution of the equation (\ref{eq-choosing-alp-posterior}), we deduce that
		\begin{equation}
			\Vert \mathbf{K} x_\alpha-\mathbf{K} x^{\ddagger} \Vert_v
			\leq \sqrt{ B_v}\left(\tau+1\right)\delta.
			\label{Kv} 
		\end{equation}

		We denote $\omega=(\omega_\lambda)_{\lambda\in\Lambda}$ with $\omega_\lambda=[\varphi(\kappa_\lambda^2)]^{-1/2}\langle x^\ddagger,u_\lambda\rangle$. Using  the definition of $\Theta$ function, a bound of frame $\mathbf{u}$ and the definition of $\mathbf{M}_{\varphi, E}$, we get 
		\begin{align}
			\Theta \left( \dfrac{\left\| \left[ r_\alpha\left(\mathbf{K}^* \mathbf{K}\right)\right]^{\frac{1}{2}} x^{\ddagger} \right\|_u^2}{\|\omega\|_2^2}\right) 
			= \Theta \left( \dfrac{\sum_{\lambda \in \Lambda} \left|\langle \left[ r_\alpha\left(\kappa^2_\lambda \right)\right]^{\frac{1}{2}} x^{\ddagger}, u_\lambda \rangle \right|^2}{\|\omega\|_2^2} \right) =  \Theta \left( \dfrac{\sum_{\lambda \in \Lambda} r_\alpha\left(\kappa^2_\lambda \right) \varphi\left( \kappa^2_\lambda \right)\left| \omega_\lambda \right|^2}{\|\omega\|_2^2} \right). 
			\nonumber
		\end{align} 
		Next, by the convexity of  $\Theta$, the Jensen inequality, and the inequality  (\ref{ieq-varrho-r}) there holds that 
		\begin{align*}
			\Theta \left( \dfrac{\sum_{\lambda \in \Lambda} r_\alpha\left(\kappa^2_\lambda \right) \varphi\left( \kappa^2_\lambda \right)\left| \omega_\lambda \right|^2}{\|\omega\|_2^2} \right) 
			& \leq   \dfrac{\sum_{\lambda \in \Lambda} \Theta \left( r_\alpha\left(\kappa^2_\lambda \right) \varphi\left( \kappa^2_\lambda \right)\right) \left| \omega_\lambda \right|^2}{\|\omega\|_2^2}  \nonumber \\
			& \leq   \dfrac{\sum_{\lambda \in \Lambda} \kappa^2_\lambda  r^2_\alpha\left(\kappa^2_\lambda \right) \varphi\left( \kappa^2_\lambda \right) \left| \omega_\lambda \right|^2}{\|\omega\|_2^2}  \nonumber \\
			& \leq   \dfrac{\sum_{\lambda \in \Lambda} \kappa^2_\lambda  r^2_\alpha\left(\kappa^2_\lambda \right) \varphi\left( \kappa^2_\lambda \right) \left| \omega_\lambda \right|^2}{\|\omega\|_2^2} . \nonumber \\
		\end{align*} 
		Combining the results here with inequality (\ref{Kv}) and bound of frame $\mathbf{u}$, we can infer that.
		\begin{align}
			\label{ieq-varrho-del-rho3}
			\Theta \left( \dfrac{\left\| \left[ r_\alpha\left(\mathbf{K}^* \mathbf{K}\right)\right]^{\frac{1}{2}} x^{\ddagger} \right\|_u^2}{\|\omega\|_2^2}\right) 
			& \leq  \dfrac{\sum_{\lambda \in \Lambda}   \left|\langle  \kappa_\lambda  r_\alpha\left(\kappa^2_\lambda \right) x^{\ddagger}, u_\lambda \rangle \right|^2}{\|\omega\|^2}    =  \frac{\left\|\mathbf{K} x_\alpha-\mathbf{K} x^{\ddagger}\right\|_v^2}{\|\omega\|_2^2} 
			\leq  \dfrac{B_v(\tau+1)^2 \delta^2}{\|\omega\|_2^2} .
		\end{align}
		
		From the inequality $B_v\geq 1$ and the definition of the source-set $\mathbf{M}_{\varphi, E}$, we deduce that  $\sqrt{B_vA_v^{-1}}(\tau+1) E \geq \Vert \omega \Vert_u$. Using the monotonicity of  $\varphi^{-1}$, a relation  $\varphi^{-1}(\lambda)=\frac{1}{\lambda} \Theta(\lambda)$ and the estimate  (\ref{ieq-varrho-del-rho3}), we obtain
		$$
		\begin{aligned}
			&\varphi^{-1}\left(  \dfrac{\left.\| r_\alpha\left(\mathbf{K}^* \mathbf{K}\right)\right]^{\frac{1}{2}} x^{\ddagger} \|_u^2}{A_v^{-1}B_v (\tau+1)^2 E^2}\right)  \leq \varphi^{-1}\left( \dfrac{\left\|\left[r_\alpha\left(\mathbf{K}^* \mathbf{K}\right)\right]^{\frac{1}{2}} x^{\ddagger}\right\|^2}{\|\omega\|_2^2}\right) \\
			& = \frac{\|\omega\|_2^2}{\left\|\left[r_\alpha\left(\mathbf{K}^* \mathbf{K}\right)\right]^{1 / 2} x^{\ddagger}\right\|_2^2} 
			\Theta \left(  \frac{\left\|\left[r_\alpha\left(\mathbf{K}^* \mathbf{K}\right)\right]^{\frac{1}{2}} x^{\ddagger}\right\|_u^2}{\|\omega\|_2^2}\right) \\
			& = \dfrac{\|\omega\|_2^2}{\left\|\left[r_\alpha\left(\mathbf{K}^* \mathbf{K}\right)\right]^{1 / 2} x^{\ddagger}\right\|_u^2} \dfrac{B_v(\tau+1)^2 \delta^2}{\|\omega\|_2^2}  
			=\dfrac{B_v( \tau+1)^2 \delta^2}{\left\|\left[r_\alpha\left(\mathbf{K}^* \mathbf{K}\right)\right]^{1 / 2} x^{\ddagger}\right\|_u^2}.
		\end{aligned}
		$$
		Equivalently,
		$$
		\Theta \left(\dfrac{\left\|\left[r_\alpha\left(\mathbf{K}^* \mathbf{K}\right)\right]^{\frac{1}{2}} x^{\ddagger}\right\|_u^2}{A_v^{-1}B_v(\tau+1)^2 E^2}\right) \leq \frac{A_v\delta^2}{E^2} .
		$$
		From here, it follows
		$$
		\left\|\left[r_\alpha\left(\mathbf{K}^* \mathbf{K}\right)\right]^{\frac{1}{2}} x^{\ddagger}\right\|_u^2  \leq A_v^{-1}B_v (\tau+1)^2 E^2  \Theta^{-1} \left(\frac{A_v\delta^2}{E^2}\right).
		$$
		This estimate and  (\ref{ieq-norm-inner-produc}) give us the result  (\ref{ieq-x-alp-del-x-poster-general}). That is 
		\begin{align*}
			\label{ieq-x-del-al-x-dageer}
			\Vert x^\delta_\alpha - x^\ddagger \Vert_X 
			&\leq \sqrt{A^{-1}_u B_u} (\tau+1) E \sqrt{ \Theta^{-1} \left(A_v\delta^2/E^2\right)} \leq \sqrt{A^{-1}_u B_u} (\tau+1) E \sqrt{ \Theta^{-1} \left(|\mathbf{v}|_{\inf}^2\delta^2/E^2\right)}.
		\end{align*}
	\end{proof}
	\subsection{Proof of Theorem \ref{Theo-WVD-K}}
	\begin{proof}
		(a)    Firstly, \textbf{D1} ensures this because, as we know, $(u_\lambda)_{\lambda\in\Lambda}$ forms an orthonormal basis for 
		$L^2( \mathbb{R}^2)$. Similarly, \textbf{D3} is also considered established through the identification of the relationship between $\{v_{\lambda}\}$ and $\{u_{\lambda}\}$ as in (\ref{eq-v-lambda}). Therefore, we only need to prove \textbf{D2}. Taking any $\theta_T$ belonging to the range of $K$, it is easily seen that for each $\lambda=(\lambda_D,\lambda_T) \in \Lambda $ then 
		$$supp \left(\mathcal{F} u_\lambda\right) \subset \left\{ \omega \in \mathbb{R}: 2^{\lambda_D} a_u \leq \left|\omega \right| \leq 2^{\lambda_D} b_u \right\}.$$
		From here, we deduce that $ 2^{2\lambda_D} a^2_u T^\gamma \leq \left|\omega \right|^2 T^\gamma \leq 2^{2\lambda_D}b^2_u T^\gamma$ for every $\omega\in supp \left(\mathcal{F} u_\lambda\right)$. Using the monotonicity property of the function $E_{\gamma, 1}(z)$, we obtain 
		$$ E_{\gamma, 1} \left(-2^{2\lambda_D} b^2_u T^\gamma\right) \leq E_{\gamma, 1} \left(- \left|\omega \right|^2 T^\gamma \right) \leq E_{\gamma, 1} \left( -2^{2\lambda_D}a^2_u T^\gamma\right). $$
		This follows that
		$$\kappa_\lambda E^{-1}_{\gamma, 1} \left(-2^{2\lambda_D} a^2_u T^\gamma\right) \leq \left| \kappa_\lambda E^{-1}_{\gamma, 1} \left(- \left|\omega \right|^2 T^\gamma \right) \right| \leq \kappa_\lambda E^{-1}_{\gamma, 1} \left( -2^{2\lambda_D}b^2_u T^\gamma\right). $$
		Moreover, from the inequality \eqref{bound-ML},
		and $\kappa_\lambda =2^{-2\lambda_D}$, we can deduce
		\begin{equation}
			\label{ieq-frame-v}
			(1+a^2_u T^\gamma)/ \overline{c} \leq \left| \kappa_j E^{-1}_{\gamma, 1} \left(- \left|\omega \right|^2 T^\gamma \right) \right| \leq (1+b^2_u T^\gamma) / \underline{c}.
		\end{equation}
		We recall that $ u_j $ is a frame, this means that for all $\theta \in L^2 \left(\mathbb{R}\right)$, 
		$$A_u \Vert \theta \Vert^2_{L^2 \left(\mathbb{R}^2\right)} \leq \sum_{\lambda \in \Lambda} \left| \langle \theta, u_\lambda \rangle \right|^2 \leq B_u \Vert \theta \Vert^2_{L^2 \left(\mathbb{R}^2\right)} .$$
		Equivalently, 
		\begin{equation}
			A_u \Vert\mathcal{F} \theta \Vert^2_{L^2 \left(\mathbb{R}^2\right)} \leq \sum_{\lambda \in \Lambda} \left| \langle \mathcal{F}\theta, \mathcal{F}u_\lambda \rangle \right|^2 \leq B_u \Vert \mathcal{F}\theta \Vert^2_{L^2 \left(\mathbb{R}^2\right)} .
			\label{ieq-frame-u-fourier}
		\end{equation}
		Take any $\theta_T \in $ ran$K$, we show that 
		$$A_v \Vert \theta_T \Vert^2_{L^2 \left(\mathbb{R}^2\right)} \leq \sum_{\lambda \in \Lambda} \left| \langle \theta_T, v_\lambda \rangle \right|^2 \leq B_v \Vert \theta_T \Vert^2_{L^2 \left(\mathbb{R}^2\right)} .$$
		This is equivalent to proving
		\begin{equation}
			\label{ieq-frame-v-fourier}
			A_v \Vert\mathcal{F} \theta_T \Vert^2_{L^2 \left(\mathbb{R}^2\right)} \leq \sum_{\lambda \in \Lambda} \left| \langle \mathcal{F} \theta_T, \mathcal{F} v_\lambda \rangle \right|^2 \leq B_v \Vert \mathcal{F} \theta_T \Vert^2_{L^2 \left(\mathbb{R}^2\right)} .
		\end{equation}
		And now, for every $\lambda \in \Lambda$,  from (\ref{eq-v-lambda-Fourier}), we have
		$$ \left| \langle \mathcal{F}\theta_T, \mathcal{F} v_\lambda \rangle \right|^2 = \left| \kappa_j E^{-1}_{\gamma, 1} \left(- \left|\omega \right|^2 T^\gamma \right)\right|^2 \left|  \langle \mathcal{F}\theta_T, \mathcal{F} u_\lambda \rangle \right|^2.$$
		Using the inequality (\ref{ieq-frame-v}), it leads us to
		\begin{align*}
			\left((1+ a^2_u T^\gamma)/ \overline{c} \right)^2\left|  \langle \mathcal{F}\theta_T, \mathcal{F} u_\lambda \rangle \right|^2  
			&\leq \left| \kappa_\lambda E^{-1}_{\gamma, 1} \left(- \left|\omega \right|^2 T^\gamma \right)\right|^2 \left|  \langle \mathcal{F}\theta_T, \mathcal{F} u_\lambda \rangle \right|^2 \\
			&\leq \left((1+ b^2_u T^\gamma)/ \underline{c}\right)^2 \left|  \langle \mathcal{F}\theta_T, \mathcal{F} u_\lambda \rangle \right|^2, 
		\end{align*} 
		for every $\lambda\in\Lambda$.
		Hence, 
		\begin{align*}
			\left((1+ a^2_u T^\gamma)/ \overline{c} \right)^2 \sum_{\lambda \in \Lambda} \left|  \langle \mathcal{F}\theta_T, \mathcal{F} u_\lambda \rangle \right|^2  
			&\leq \sum_{\lambda \in \Lambda} \left| \langle \mathcal{F}\theta_T, \mathcal{F} v_\lambda \rangle \right|^2   \leq \left(1+ b^2_u T^\gamma/ \underline{c}\right)^2 \sum_{\lambda \in \Lambda} \left|  \langle \mathcal{F}\theta_T, \mathcal{F} u_\lambda \rangle \right|^2. 
		\end{align*}
		Combining with (\ref{ieq-frame-u-fourier}), we get 
		$$ \left((1+ a^2_u T^\gamma)/ \overline{c} \right)^2 A_u \Vert \mathcal{F}\theta_T \Vert^2  \leq \sum_{\lambda \in \Lambda} \left| \langle \mathcal{F}\theta_T, \mathcal{F} v_\lambda \rangle \right|^2  \leq \left((1+ b^2_u T^\gamma)/ \underline{c}\right)^2  B_u \Vert  \mathcal{F}\theta_T \Vert^2.$$
		Finally, (\ref{ieq-frame-v-fourier}) be proved for $A_v = \left( (1+a^2_u T^\gamma)/ \overline{c} \right)^2 A_u $ and $B_v = \left( (1+b^2_u T^\gamma)/ \underline{c}\right)^2  B_u$.
		
		(b) We find the source condition for the solution $\theta_0$. 
		Naturally, we can assume that $\theta_0\in H^p(\mathbb{R})$
		for $p\geq 0$. Putting
		\begin{equation}
			G({\lambda_D})=\{\omega\in\mathbb{R}: 2^{\lambda_D}a_u\leq |\omega|\leq 2^{\lambda_D}b_u\},
			\label{setG}
		\end{equation}  
		we note that ${\rm supp} u_{\lambda}\subset G(\lambda_D)$. For $\lambda=(\lambda_D,\lambda_T)$,  we can write
		\begin{align*}
			\langle \theta_0,u_\lambda\rangle &= \frac{1}{2\pi}
			\langle \mathcal{F}\theta_0,\mathcal{F}u_\lambda\rangle
			=\frac{1}{2\pi}\int_{\mathbb{R}}\mathcal{F}\theta_0(\xi)\overline{\mathcal{F}u_\lambda(\xi)}d\xi = \frac{1}{2\pi}\int_{\mathbb{R}}1_{G(\lambda_D)}(\omega)\mathcal{F}\theta_0(\xi)\overline{\mathcal{F}u_\lambda(\xi)}d\xi.
		\end{align*}
		For $\lambda_D>0$, using the Bessel inequality yields
		\begin{align*}
			\sum_{\lambda_T\in\mathbb{Z}}|\langle \theta_0,u_\lambda\rangle|^2\leq
			\frac{1}{2\pi} \|1_{G(\lambda_D)}\mathcal{F}(\theta_0)\|^2
			\leq 2^{-2p\lambda_D}\omega_{\lambda_D}^2=\kappa_{\lambda}^p\omega_{\lambda_D}^2,
		\end{align*}
		where
		$$  w_{\lambda_D}=a_u^{-p}\|1_{G(\lambda_D)}(1+\omega^2)^p\mathcal{F}(\theta_0)\|. $$ 
		For $\lambda_D\leq 0$, we have $\kappa_\lambda=1$ and
		$$ \langle \theta_0,u_\lambda\rangle=\kappa_\lambda^{p/2}\langle \theta_0,u_\lambda\rangle. $$
		Put $w_\lambda=\|u_0\|w'_\lambda $ for $\lambda_D>0$ and $w_\lambda=\langle \theta_0,u_\lambda\rangle$
		for $\lambda_D\leq 0$. Direct computations yields
		$$\sum_{\lambda_D\in \mathbb{Z}}\kappa_\lambda^{-p} |\langle\theta_0, u_\lambda\rangle|^2 = \sum_{\lambda\in\mathbb{Z}} |w_\lambda|^2\leq C\|\theta_0\|^2_{H^r}. $$
		So the function $\theta_0\in\mathbf{M}_{\varphi,E}$ where 
		$\varphi(\mu)=\mu^{p/2}$ and $E$ is large enough.
		
		(c) To obtain the order optimal result, we verify the condition in Theorem \ref{Theo:lower-bound}. We have 
		$\delta^*_\lambda=|\mathbf{v}|_{\inf}^{-1}E\sqrt{\kappa_\lambda^2\varphi(\kappa_\lambda^2)}
		=|\mathbf{v}|_{\inf}^{-1}E
		\sqrt{\kappa_\lambda^{2+p}}=|\mathbf{v}|_{\inf}^{-1}E
		2^{-\lambda_D(2+p)}$. Letting $0<\delta<|\mathbf{v}|^{-1}_{\inf}E2^{-(2+p)}$, we can choose a $\lambda_\delta$ such that 
		$$ \delta_{\lambda_\delta}^*= |\mathbf{v}|_{\inf}^{-1}E
		2^{-\lambda_\delta(2+p)}\leq \delta\leq |\mathbf{v}|_{\inf}^{-1}E
		2^{-(\lambda_\delta-1)(2+p)}=2^{(2+p)}\delta_{\lambda_\delta}^*. $$
		So we have $\delta\in \cup_{\lambda\in \Lambda}D_{\lambda,2^{-2(1+2p) }}$. 
		
	\end{proof}

	\subsection{Proof of Theorem \ref{fractional-backward}}
	
	The fact that the function $g_\alpha$ satisfies Assumptions  A, B, C is a known result. However, for the convenience of the reader, we will check these assumptions.
	
	(a) We verify that $g_\alpha(\mu)=\frac{1}{\alpha+\mu}$ and
	$\varphi(\mu)=\mu^{p/2}$ satisfy Assumptions C, A1, A2. Direct verifying yields that Assumption C holds for $g_\alpha$. The index function $\varphi$ satisfies Assumption A1. We verify Assumption A2. 
	We have 
	$$  \sqrt{\mu}g_\alpha(\mu)=\frac{\sqrt{\mu}}{\alpha+\mu}
	\leq \frac{2\sqrt{\mu}}{\sqrt{\alpha\mu}}=\frac{2}{\sqrt{\alpha}}. $$
	We verify Assumption A2 (ii). We have
	\begin{align*}
		|1-\mu g_\alpha(\mu)|\sqrt{\varphi(\mu)} &=
		\frac{\alpha\mu^{p/4}}{\alpha+\mu}.
	\end{align*}
	Put $H(\mu)=\frac{\alpha\mu^r}{\alpha+\mu}$, $r\in (0,1)$. We have
	$H'(r)=\alpha\frac{r\mu^{r-1}(\alpha+\mu)-\mu^r}{(\alpha+\mu)^2}$. The function attains its maximum when $r(\alpha+\mu)-\mu=0$ which gives $\mu=\frac{\alpha}{1-r}$. Choose $r=p/4$, we obtain 
	Assumption A2 (ii) 
	Hence $H(\mu)\leq C\alpha^{r}$. For $p=4$ we have $r=1$, $H(\alpha)\leq \alpha$ which give Assumption A2 (ii). 
	
	(b) We first consider Assumption B1. In fact we have 
	$1-\mu g_\alpha(\mu)=\frac{\alpha}{\alpha+\mu}\to 1$ as $\alpha\to\infty$. Hence Assumption B1 (i) holds. The function $g_\alpha(.)$ is continuous with respect to $\alpha$. Hence Assumption B1 (ii) holds.
	
	We verify Assumption B2. In fact we have $1-\mu g_\alpha(\mu)=\frac{\alpha}{\alpha+\mu}$. We also have
	$\ell_\alpha=\sup_{\mu\geq 0} g_\alpha(\mu)=\frac{1}{\alpha}$ which satisfies Assumption B2 (iii)
	with $\ell_*=\ell^*=1$.
	Finally, we verify that the function $\varphi$ is concave. In fact, we have $\varphi''(\mu)=(p/2)(p/2-1)\mu^{p/2-2}<0$ since $0<p\leq 2$. 
	
	(c) We have $\kappa_{\lambda,N}=e^{-NT}$, $\varphi(\mu)=
	(-\ln \mu)^{-p}$. This follows that $\varphi(\kappa^2_{\lambda,N})
	=(2NT)^{-p}$. Hence, from \eqref{set-D}, we obtain
	$$\delta_{\lambda,N}^*=|\mathbf{v}|_{\inf}^{-1}
	E\sqrt{\kappa_{\lambda,N}^2\varphi(\kappa_{\lambda,N}^2)}
	=|\mathbf{v}|_{\inf}^{-1}\sqrt{(2NT)^{-p}e^{-2NT}}.$$
	Hence, for every $\delta\in (0,\delta_0)$ where 
	$\delta_0=|\mathbf{v}|_{\inf}^{-1}\sqrt{(2T)^{-p}e^{-2T}}$
	we can find $N_0\in\mathbb{N}$ such that
	$\delta_{\lambda,N_0}^*\leq\delta\leq \delta_{\lambda,N_0-1}^* $. We note that
	\begin{align*}
		\delta_{\lambda,N_0-1}^*=
		|\mathbf{v}|_{\inf}^{-1}\sqrt{(2(N_0-1)T)^{-p}e^{-2(N_0-1)T}}\leq 
		e^T|\mathbf{v}|_{\inf}^{-1}\sqrt{(2N_0T)^{-p}e^{-2N_0T}}=e^T\delta_{\lambda,N_0}^*.
	\end{align*}
	Hence $\delta\in [\delta_{\lambda,N_0}^*,\beta^{-1} \delta_{\lambda,N_0}^*]\subset \cup_{(\lambda,N)\in\Lambda\times\mathbb{N}}D_{(\lambda,N),\beta}$ with $\beta=e^{-T}$.

	\subsection{Proof of Theorem \ref{Theo-WVD-K1}}
	
	(a) We claim that $\{u_{\lambda,N}\}$ is a frame over $L^2(\mathbb{R})$. For $\theta\in L^2(\mathbb{R})$, we have
	\begin{align}
		\langle \theta,u_{\lambda,N}\rangle=\frac{1}{2\pi}
		\langle \mathcal{F}(\theta),\mathcal{F}u_{\lambda,N}\rangle=\frac{1}{2\pi}
		\langle 1_{G(\lambda_D)}1_{B_N}\mathcal{F}(\theta), \mathcal{F}u_{\lambda}\rangle.     
	\end{align}
	Hence, using the Bessel inequality gives
	\begin{align*}
		\sum_{N=0}^\infty\sum_{\lambda\in\Lambda}|\langle \theta,u_{\lambda,N}\rangle|^2
		&=\frac{1}{2\pi}
		\sum_{N=0}^\infty\sum_{\lambda\in\Lambda}
		| \langle \mathcal{F}(\theta), \mathcal{F}u_{\lambda}\rangle|^2 \leq\frac{1}{2\pi}\sum_{N=0}^\infty \|1_{B_N}\mathcal{F}(\theta)\|^2=\frac{1}{2\pi}\|\mathcal{F}(\theta)\|^2=
		\|\theta\|^2.   
	\end{align*}
	This follows that $\{u_{\lambda,N}\}$ is a tight frame. We find $v_{\lambda,N}$ and $\kappa_{\lambda,N}$ such that
	$$ \mathcal{F} v_{\lambda,N}=\kappa_\lambda \exp(|\omega|^2T)\mathcal{F}(u_{\lambda,N})=
	\kappa_\lambda 1_{B_N} \exp(|\omega|^2T)\mathcal{F}(u_{\lambda}).  $$
	For $\omega\in B_N$, we have
	$  \exp(NT)\leq  \exp(|\omega|^2T)\leq \exp((N+1)T). $
	Hence, we can choose $\kappa_{\lambda,N}=e^{-N T}$ and obtain
	\begin{equation}
		1\leq \kappa_{\lambda,N} \exp(|\omega|^2T)\leq e^T~~~~ {\rm for}~\omega\in B_N.
		\label{singular-value}	
	\end{equation} 
	We verify that $\{v_{\lambda,N}\}$ is a frame. We have
	\begin{align*}
		\sum_{N=0}^\infty\sum_{\lambda\in\Lambda}|\langle \theta,v_{\lambda,N}\rangle|^2
		&=\frac{1}{2\pi}
		\sum_{N=0}^\infty\sum_{\lambda\in\Lambda}
		| \langle \kappa_{\lambda,N}e^{|\omega|^2T}1_{B_N}\mathcal{F}(\theta), \mathcal{F}u_{\lambda}\rangle|^2 =\frac{1}{2\pi}\sum_{N=0}^\infty \| \kappa_{\lambda,N}e^{|\omega|^2T}1_{B_N}\mathcal{F}(\theta)\|^2.   
	\end{align*}
	Using \eqref{singular-value} gives 
	\begin{align*}
		\frac{1}{2\pi}\sum_{N=0}^\infty \|1_{B_N}\mathcal{F}(\theta)\|^2
		\leq
		\frac{1}{2\pi}\sum_{N=0}^\infty \| \kappa_{\lambda,N}e^{|\omega|^2T}1_{B_N}\mathcal{F}(\theta)\|^2
		\leq \frac{e^T}{2\pi}\sum_{N=0}^\infty \| 1_{B_N}\mathcal{F}(\theta)\|^2.	
	\end{align*}
	But $\frac{1}{2\pi}\sum_{N=0}^\infty \|1_{B_N}\mathcal{F}(\theta)\|^2=\frac{1}{2\pi} \|\mathcal{F}(\theta)\|^2=\|\theta\|^2 $. Hence $\{v_{\lambda,N}\}$ is a tight frame.
	
	(b) We find the source set of our problem. 
	Here the set $G(\lambda_D)$ is defined in \eqref{setG}. For $\theta_0\in H^p(\mathbb{R})$, we have
	We find the source condition for the solution $\theta_0$. 
	As in the previous theorem, we can assume that $\theta_0\in H^p(\mathbb{R})$
	for $p\geq 0$. For $(\lambda,N)=(\lambda_D,\lambda_T,N)$,  we can write
	\begin{align*}
		\langle \theta_0,u_{\lambda,N}\rangle &= \frac{1}{2\pi}
		\langle \mathcal{F}\theta_0,\mathcal{F}u_\lambda\rangle
		=\frac{1}{2\pi}\int_{\mathbb{R}}1_{G(\lambda_D)}1_{B_N}\mathcal{F}\theta_0(\xi)\overline{\mathcal{F}u_\lambda(\xi)}d\xi.
	\end{align*}
	Hence, for $\lambda_D>0$, using the Bessel inequality yields
	\begin{align*}
		\sum_{\lambda_T\in\mathbb{Z}}	|\langle \theta_0,u_{\lambda,N}\rangle|^2&\leq
		\|1_{G_D}1_{B_N}\mathcal{F}\theta_0\|^2
		\leq (1+2^{2\lambda_D}a_u)^{-p} \omega_{\lambda_D,N}^2,
	\end{align*}
	where
	$$  \omega_{\lambda_D,N}=\|1_{G_D}1_{B_N}
	(1+\omega^2)^{p}\mathcal{F}\theta_0\|. $$
	This follows that
	$$ \sum_{\lambda_T\in\mathbb{Z}}	N^{p}|\langle \theta_0,u_{\lambda,N}\rangle|^2\leq
	N^{p}\|1_{G_D}1_{B_N}\mathcal{F}\theta_0\|^2
	\leq N^{p}(1+2^{2\lambda_D}a_u)^{-p} \omega_{\lambda_D,N}^2. $$
	On the other hand, we have $\langle\theta_0,u_{\lambda,N}\rangle\not=0$ then $B_N\cap G_{\lambda_D}\not=\emptyset$, which gives $\sqrt{N}\leq 2^{\lambda_D}b_u $.
	So we have
	$$ \sum_{\lambda_T\in\mathbb{Z}}	N^{p}|\langle \theta_0,u_{\lambda,N}\rangle|^2
	\leq 2^{2p\lambda_D}b_u^{2p}(1+2^{2\lambda_D}a_u)^{-p} \omega_{\lambda_D}^2\leq b_u^{2p}a_u^{-p} \omega_{\lambda_D,N}^2. $$
	Noting that $N=-\frac{1}{2T}\ln \kappa^2_{\lambda,N}$ we have
	$N^p=\frac{1}{(2T)^p}[\varphi(\kappa_\lambda^2)]^{-1}$ and
	$$ \sum_{\lambda_T\in\mathbb{Z}}[\varphi(\kappa_\lambda^2)]^{-1}|\langle \theta_0,u_{\lambda,N}\rangle|^2
	\leq (2T)^pb_u^{2p}a_u^{-p}  \omega_{\lambda_D,N}^2. $$
	This implies that
	$$ \sum_{N=0}^\infty\sum_{\lambda_D\in \mathbb{Z}}\sum_{\lambda_T\in\mathbb{Z}}[\varphi(\kappa_\lambda^2)]^{-1}|\langle \theta_0,u_{\lambda,N}\rangle|^2
	\leq (2T)^pb_u^{2p}a_u^{-p}  \sum_{N=0}^\infty\sum_{\lambda_D\in \mathbb{Z}}
	\omega_{\lambda_D,N}^2
	=(2T)^pb_u^{2p}a_u^{-p}\|\theta_0\|^2_{H^p}. $$
	So we obtain 
	$\theta_0\in \mathbf{M}_{\varphi,E}$ where $E=(2T)^{p/2}b_u^{p}a_u^{-p/2}\|\theta_0\|^2_{H^p}  $, $\varphi(\mu)=(-\ln\mu)^{-p}$ for $\mu\in (0,1)$. 
	
	\subsection{Proof of Theorem \ref{classical-backward}}
	
	We verify that $g_\alpha(\mu)=\frac{1}{\alpha+\mu}$ and
	$\varphi(\mu)=\left(-\ln\mu\right)^{-p}$ satisfy Assumptions C, A1, A2. As shown in the proof of Theorem
	\ref{fractional-backward}, Assumption C holds for $g_\alpha$. The index function $\varphi$ satisfies Assumption A1 (i), (ii). Using Theorem 9.1 in \cite{Tautenhahn} gives that the function $\Theta$ is convex on $(0,\infty)$ for $p>0$. 
	
	Assumption A2 (i) is verified in the proof of Theorem \ref{fractional-backward}. 
	We verify Assumption A2 (ii). We have
	\begin{align*}
		|1-\mu g_\alpha(\mu)|\sqrt{\varphi(\mu)} &=
		\frac{\alpha\left(-\ln\mu\right)^{-p/2}}{\alpha+\mu}.
	\end{align*}
	Putting $\tau=\alpha/\mu$, we obtain 
	$$ \frac{\alpha\left(-\ln\mu\right)^{-p/2}}{\alpha+\mu}=\frac{\tau(\ln(\tau \alpha^{-1}))^{-p/2}}{\tau+1}.  $$
	For $\alpha<\tau\leq\sqrt{\alpha}$, $\eta\in (0,1)$, using the inequality $z^\eta(\ln z)^{-p/2}\to 0 $ as $z\to 0$, we have
	$$  \frac{\tau(\ln(\tau \alpha^{-1}))^{-p/2}}{\tau+1}\leq 
	\frac{\tau^{1-\eta}\alpha^{\eta} (\tau\alpha^{-1})^{\eta}(\ln(\tau \alpha^{-1}))^{-p/2}}{\tau+1}\leq C\alpha^\eta\leq C' (-\ln \alpha)^{-p/2}.     $$
	For $\tau\geq\sqrt{\alpha}$ we have
	$$  \frac{\tau(\ln(\tau \alpha^{-1}))^{-p/2}}{\tau+1}\leq  
	(\ln(\tau \alpha^{-1}))^{-p/2}\leq (\ln( \alpha^{-1/2}))^{-p/2}\leq C\ln( \alpha^{-1})^{-p/2}. $$
	Using Theorem \ref{Theo-WVD-K1} (ii), we can find $E_1>0$ such that
	$\mathbf{M}_{E_1}:=\{\theta_0\in H^{p}(\mathbb{R}): \|\theta_0\|_{H^p}\leq E_1\}\subset \mathbf{M}_{\varphi,E}$. So we have
	$ \Omega\left(\mathbf{M}_{E_1}, \delta \right)\leq 
	\Omega\left(\mathbf{M}_{\varphi,E}, \delta \right) $.
	Using the classical results in \cite{Hohage, Tautenhahn},
	we can find a $C_0>0$ such that
	$ \Omega\left(\mathbf{M}_{E_1}, \delta \right)\geq 
	C_0\ln( E_1\delta^{-1})^{-p}.  $
	Hence 
	$$ \inf_R\Delta(\mathbf{M}_{\varphi, E},\delta,\mathbf{R})\geq \Omega\left(\mathbf{M}_{\varphi,E}, \delta \right)\geq 
	C_0\ln( E_1\delta^{-1})^{-p}\geq C'E\ln( E\delta^{-1})^{-p} .  $$
	So, we obtain the order optimal property of our regularization.

	\section{Conclusion}
	The paper has investigated DFD regularizations in both a priori and a posteriori cases. For the case where the $\{u_\lambda\}$ system is minimal, we have proved the seuqntial order optimality property  and the global optimality for DFD regularizations.  Some issues that need to be investigated in the future are
	
	- Methods of constructing DFDs for ill-posed problems
	
	- Investigation of the relationship between the classical source condition and the DFD source condition.
	
	- Investigation of optimality in the case where $\{u_\lambda\}$ is not minimal.
	
	- Find the condition of the DFD singular value so that the regularization method is uniformly optimal
	
	\textbf{Acknowledgments.} This research received support from Vietnam National University (VNU-HCM) under grant number T2024-18-01. The first author also completed a portion of this paper during a sponsored visiting research period at Vietnam Institute for Advanced Studies in Mathematics (VIASM) from July 1 to August 30, 2025.
	

\end{document}